\documentclass[11pt,reqno]{amsart}

\usepackage{amsmath,amssymb}
\usepackage{amsfonts}
\usepackage{xcolor}
\usepackage{mathtools}
\usepackage{nccmath}

\usepackage{hyperref}
\usepackage{comment}
\usepackage[mathscr]{euscript}
\DeclareSymbolFont{rsfs}{U}{rsfs}{m}{n}
\DeclareSymbolFontAlphabet{\mathscrsfs}{rsfs}

\hypersetup{
    colorlinks=true,
    linkcolor=red,
    filecolor=magenta,      
    urlcolor=cyan,
}
\usepackage{geometry}
\def\XXint#1#2#3{{\setbox0=\hbox{$#1{#2#3}{\int}$ }
\vcenter{\hbox{$#2#3$ }}\kern-.6\wd0}}

\usepackage{amsthm}
\newtheorem{theorem}{Theorem}[section]

\newtheorem{lemma}[theorem]{Lemma}
\newtheorem{prop}[theorem]{Proposition}
\newtheorem{rem}[theorem]{Remark}
\newtheorem{claim}[theorem]{Claim}

\numberwithin{equation}{section}

\newcommand{\la}{\langle}
\newcommand{\ra}{\rangle}

\renewcommand{\Im}{\operatorname{Im}}

\newcommand{\R}{\mathbb R}
\newcommand{\C}{\mathbb C}

\newcommand{\N}{\mathbb N}
\newcommand{\X}{\mathcal X}

\numberwithin{equation}{section}

\author[I. Petrenko]{Iryna Petrenko}
\address{Department of Mathematics  \& Statistics\\Florida International University,  Miami, FL, USA}
\curraddr{}
\email{ipetr006@fiu.edu}

\author[O. Ria\~{n}o]{Oscar Ria\~{n}o }
\address{Universidad Nacional de Colombia, Bogot\'a, Colombia}
\curraddr{}
\email{ogrianoc@unal.edu.co}

\author[S. Roudenko]{Svetlana Roudenko}
\address{Department of Mathematics  \& Statistics\\Florida International University,  Miami, FL, USA}
\curraddr{}
\email{sroudenko@fiu.edu}

\title[Local well-posedness for bi-NLS]{Local well-posedness for the bi-harmonic\\ nonlinear Schr\"{o}dinger equation in low dimensions} 



\begin{document}

\begin{abstract} 
In this paper we consider the 4th order (or biharmonic) nonlinear Schr\"odinger  (NLS) equation in dimensions 1, 2, and 3, where the potential term is expressed as a power nonlinearity (for any positive power) and the dispersion operator has the {\it fourth} order combined with the lower {\it second} order. The fourth order NLS equation has recently attracted attention of researchers since quartic solitons have been experimentally obtained in optics, and thus, the mathematical theory of solutions to the fourth-order NLS equation in physical dimensions $N=1,2,3$ is timely to develop. 
In this work we show the local well-posedness of the 4th order or biharmonic NLS equation on a weighted subset of a Sobolev space as well as in $H^s$ spaces.
\end{abstract}

\maketitle

{\hypersetup{linkcolor=black}%
\tableofcontents
}

\section{Introduction}

In this paper we are interested in obtaining local well-posedness theory for the general biharmonic NLS (bi-NLS) equation that includes not only the bi-Laplacian $\Delta^2$ but also lower order derivatives, namely, 
\begin{equation}\label{biNLS}
i\partial_{t} {u} -\Delta^{2} {u} -2a \Delta u +\lambda|u|^{\alpha}u=0, \quad x \in \mathbb{R}^N, \quad \lambda\in \R, \quad a \in \mathbb{R},
\end{equation}
which is becoming important in physical applications such as quartic solitons or light bullets in nonlinear optics, see \cite{Lightbullets2024}, \cite{Nature}. 

The biharmonic NLS equation appeared in early 90s in works of Karpman \cite{K1991} and Karpman-Shagalov \cite{KS1991}, where they incorporated the higher-order dispersion that appeared in laser beam propagation. A recent revival of this equation was due to the experimental work that produced pure quartic solitons in  photonic crystals, for example, see \cite{Nature}, \cite{Tam2019}, \cite{Tam2020}. For a review refer to \cite{Aceves2023}, and some further developments, e.g., see \cite{HA2024}, \cite{PA2024}.

The local well-posedness of the initial value problem  (or Cauchy problem) for \eqref{biNLS} with $u(0,x)=u_0$ in some Sobolev space has been studied for some time. In the energy space $H^2$ it was established by Ben-Artzi, Koch \& Saut \cite{BAKS2000}, where the sharp dispersive estimates were given for the biharmonic Schr\"odinger operator; for the case $a=0$ see also Kenig, Ponce \& Vega \cite{KPV1991}. More recent results for the biharmonic NLS were established by Pausader \cite{Pausader2007}, \cite{Pausader2009_1}, \cite{Pausader2009_2}, see also Dinh \cite{Dinh2018}. In particular, the result of the above papers was the local well-posedness in $H^2$ in one and two dimensions for $\alpha \geq 1$, in dimensions $N=3,4$ for $\alpha \geq 2/N$, and in dimensions $N \geq 5$ for $2/N \leq \alpha < 8/(N-4)$. 
Hao, Hsiao and Wang \cite{HHW2006, HHW2007} discussed the Cauchy problem in a high-regularity setting. 
Hayashi, Mendez-Navarro, Naumkin \cite{HMN2016} studied local and global well-posedness and scattering in weighted Sobolev spaces. 
Typically the local well-posedness is obtained by such standard machinery as Strichartz estimates, however, that approach does not allow considering {\it small powers} in higher dimensions, or is quite challenging, if at all possible, for {\it combinations} of dispersions  or nonlinearities. Another approach to obtain well-posedness, that has been introduced by Cazenave-Naumkin in \cite{CazNaum2016}, is in a weighted subset of Sobolev space with some non-vanishing condition. This approach turned out to be quite useful in various equations, for example, for the NLS with combined nonlinearities \cite{Alex}, for the Hartree equation \cite{ARR2022}, for periodic NLS \cite{SRR}, for the KdV 
\cite{LinaresMiyazaki2019}, as well as an extension to the general powers in \cite{Diana2023}. 
Local well-posedness and scattering in the fourth order NLS have been recently considered in \cite{HMN2016}, where authors introduce a factorization when $\alpha>3$ for the free evolution group to obtain the large time asymptotics of solutions. In \cite{Naum2022} the defocusing pure case of biharmonic NLS is considered and more rapid time decay rate is obtained, compared to the corresponding linear case. 
\smallskip

In this paper we investigate the local well-posedness of \eqref{biNLS} in physical dimensions 1, 2, and 3 on a weighted subspace of the Sobolev space $H^1$, where we can address all positive powers of nonlinearity $\alpha > 0$ (including low powers $0<\alpha<1$ in dimensions one and two, and $0<\alpha <2/3$ in three dimensions), as well as the slow decay initial conditions (such as polynomial decay).  
Inspired by the results in \cite{CazNaum2016}, we introduce a space of functions $\X$ (denoted as $\mathcal X_N$ in each dimension $N$), which is convenient for dealing with nonsmooth nonlinearities. 
\smallskip

We start with the definition of the space in 1D. For a given nonlinearity $\alpha$, take $n \in \mathbb N$ such that
\begin{equation}\label{E:n-1d}
{n>0}   
\end{equation}
and set
\begin{equation}\label{SpaceDim1gen}
r = \max\{5, 3 n+2 \} \quad \mbox{and} \quad M =  3 n +r. 
\end{equation}
Define the space $\X$ on $\mathbb R$, denoted by $\X_1$, with parameters $n,r,M$ as follows
\begin{equation}\label{spaceX1-1D}
\X_1=
\left\{
\begin{aligned}
& \, u \in {  \dot{H}^{1}(\R)\cap \dot{H}^{M}(\mathbb{R})}, \\
& \langle x \rangle^{n} u \in L^{\infty}(\mathbb{R}), \quad  \\
& \langle x \rangle^{n} \partial^{\beta}u \in L^{2}(\mathbb{R}), \quad \mbox{for} \quad  1  \leq \beta  \leq r,\\
& \, \partial^{\beta}u \in L^{2}(\mathbb{R}), \hspace{0.4in} \mbox{for} {  \quad r+1 \leq \beta \leq M},
\end{aligned}
\right.
\end{equation}
which is equipped with the norm
\begin{equation}\label{Xnorm1D}
 \Vert u\Vert_{\X_1} =  \Vert\langle x \rangle^{n}  u \Vert_{L^{\infty}(\R)} + \sum_{1 \leq \beta \leq r} \Vert\langle x \rangle^{n}\partial^{\beta}u\Vert_{L^{2}(\R)} +{ \sum_{r+1\leq \beta\leq M} \Vert \partial^{\beta}u\Vert _{L^{2}(\R)}}.  
\end{equation}
It seems that the space $\mathcal X_1$ should be denoted by $\X_1^{(n,r,M)}$, however, we typically suppress the dependence on the parameters $n,r,M$, see the next remark. 
\begin{rem}\label{R:rnM}
Observe that $\X_1^{(n_1, r_1,M_1)} \supseteq \X_1^{(n_2,r_2,M_2)}$ if $n_1 \leq n_2$, {  and} $r_1 \leq r_2$,  and $M_1 \leq M_2$; in other words, $\|f\|_{\mathcal X_1^{n_1,r_1,M_1}} \leq \|f\|_{\X^{n_2,r_2,M_2}}$. 
Therefore, if we show boundedness for the smaller indexes of the space $\mathcal X_1$, it will also hold for the larger ones. 
\end{rem}

\begin{rem}\label{R:r}
In this paper, we set the index $r$ with the equality sign; in general, one can be flexible and choose $r\geq \max$ of given quantities, however, then choosing $M$ will be delicate, and will need interpolations. We investigate this elsewhere. 
\end{rem}

In dimensions $N=2$ and $3$, we define $\X_N$
as follows. For a given nonlinearity $\alpha$, take $n\in\N$ such that 
{ \begin{equation}\label{E:n-2D-3D}
n > \frac{N}{2\alpha}=\left\{
\begin{aligned}
& \,  \frac1{\alpha} & \mbox{if}~~ N=2\\
& \,  \frac3{2\alpha} & \mbox{ if}~~ N=3.
\end{aligned}
\right.
\end{equation}}
Then choose $r,M\in\N$ such that
\begin{equation}\label{spaceX1 gen 2D}
r=\max\{{  7}, 3 n+4\}\quad \mbox{and} \quad M=3 n +r, 
\end{equation}
and define  
\begin{equation}\label{spaceX1 2D}
\X_{N}=
\left\{
\begin{aligned}
& \, u \in { \dot{H}^{1}(\R^N)\cap \dot{H}^{M}(\R^N)}, \\
&\langle x \rangle^{n}\partial^{\beta}u \in L^{\infty}(\mathbb{R}^N), \quad \mbox{for each multi-index} \quad  |\beta| \leq 1,  \\
& \langle x \rangle^{n} \partial^{\beta}u \in L^{2}(\mathbb{R}^N), \quad \mbox{for each multi-index} \quad  2  \leq |\beta | \leq r,\\
& \, \partial^{\beta}u \in L^{2}(\mathbb{R}^N), \quad  \mbox{for each multi-index} \quad r+1 \leq |\beta| \leq M,
\end{aligned}
\right.
\end{equation}
equipped with the norm
\begin{equation}\label{XnormND}
\Vert u\Vert_{\X_N} = \sum_{ |\beta| \leq 1}
\Vert\langle x \rangle^{n}\partial^{\beta} u\Vert_{L^{\infty}(\R^N)} + \sum_{2 \leq |\beta| \leq r} \Vert\langle x \rangle^{n}\partial^{\beta}u\Vert_{L^{2}(\R^N)} +{  \sum_{r+1\leq |\beta|\leq M}\Vert \partial^{\beta} u\Vert _{L^{2}(\R^N)}}.  
\end{equation}

To show that the space $\mathcal X_N$ is non-empty, we give the following example: 
\begin{equation*}
u_0(x)=\frac{c}{\langle x \rangle^n}+\varphi, \quad c\in \C \setminus\{0\},
\end{equation*}
where $\varphi$ is a Schwartz class with $\|\langle x\rangle^n \varphi\|_{L^{\infty}}< |c|$.
We also mention that this function can be taken as an initial data for the next theorem (which also satisfies a non-vanishing condition \eqref{inf cond u0} below).

\smallskip

\begin{rem} The spaces $\mathcal{X}_N$ are Banach spaces. In particular, if $n>\frac{N}{2}$, then the estimate $\|u\|_{L^2}\leq \|\langle x\rangle^{-n}\|_{L^2}\|\langle x \rangle^n u\|_{L^{\infty}}<\infty$ implies that $\mathcal{X}_N\subseteq H^{M}(\mathbb{R}^N)$.
\end{rem}

Our main result is the local well-posedness of solutions to \eqref{biNLS} in the space $ \mathcal X_{N}$, $N=1,2$ or $3$. 

\begin{theorem}[Local well-posedness in $\X_N$]\label{theorem1}
Let $\alpha>0$, $\lambda\in\C$,  and $a\in \mathbb{R}$.
Let $n, r,M\in \mathbb{Z}^{+}$ be defined by \eqref{E:n-1d} and \eqref{SpaceDim1gen} in dimension $N=1$, and for dimensions $N = 2, 3$, by \eqref{E:n-2D-3D} and \eqref{spaceX1 gen 2D}. Consider the space $\X_N$ defined by \eqref{spaceX1-1D} for $N=1$, and by \eqref{spaceX1 2D} for $N=2,3$. If $u_0\in \X_N$ satisfies
\begin{equation}\label{inf cond u0}
       \inf_{x \in \mathbb{R}^N}\langle x \rangle^{n}|u_0(x)| >0 , 
   \end{equation}
then there exist a time $T>0$ and a unique solution $u\in C([-T,T];\X_N)$ of \eqref{biNLS} such that
	\begin{equation}\label{Xinf}
		 \inf_{(x,t)\in\mathbb{R}^N\times [-T,T]} \langle x \rangle^n |u(x,t)|>0.
	\end{equation}
Moreover, the map $u_0 \mapsto u(\cdot,t)$ is continuous in the following sense: for any $0<\widetilde{T}<T$, there exists a neighborhood $V$ of $u_0$ in $\X_N$, satisfying \eqref{inf cond u0}, such that the map data-to-solution is Lipschitz continuous from $V$ into the class $C([-\widetilde{T},\widetilde{T}],\X_N)$.
\end{theorem}

\begin{rem}
In the one-dimensional case $N=1$, the weight exponent $n>0$ can be chosen arbitrarily, in particular, independent of $\alpha>0$. However, for dimensions $N=2,3$, we require
$$n>\frac{N}{2\alpha}.$$
This restriction arises in the proof of Theorem \ref{theorem1}, where it is necessary to guarantee that $\langle x\rangle^{-n\alpha}\in L^2(\mathbb{R}^N)$, see Case 1 in Lemma \ref{nonlinear est} below. Such a requirement never appears in the one-dimensional setting, where the nonlinear estimates can be established without imposing this additional integrability condition. Consequently, the one-dimensional case enjoys greater flexibility with respect to the choice of the weight exponent.
\end{rem}

The proof of Theorem \ref{theorem1} 
is based on a fixed point argument, where the estimates that we derive in this paper below are fundamental to obtain these results, specifically, the linear estimates in  Lemmas \ref{pure bi-laplacian}, \ref{linear est l-laplacian}, 
and the nonlinear estimates in Lemma \ref{nonlinear est}. Given $\alpha>0$, the nonlinearity $|u|^{\alpha}u$ is not necessarily regular at the origin, therefore, we impose the condition \eqref{inf cond u0}, similarly to \cite{CazNaum2016} and construct local solutions, for which \eqref{Xinf} holds. This ensures that the nonlinearity $|u|^{\alpha}u$ is supported outside the origin, and hence, smooth. We also note that the weight in the condition \eqref{inf cond u0} is necessary, specifying that $u_0$ not only does not vanish, but also, it has a polynomial decay (note, that $\la x\ra^{-n} \lesssim |u_0(x)| \in L^2(\mathbb R^N)$ implies the restriction $n > \frac{N}2$). 
For this reason, we require a detailed analysis in weighted function spaces for the solutions of the linear equation $i\partial_t u-\Delta^{2}u=0$. More precisely, we establish the following weighted linear estimate.
\begin{lemma}\label{linearweightedspaces}
Let $N, n\in \mathbb N$. 
Then for any $t\in \mathbb{R}$
\begin{equation}\label{weightedineq-biharmonic}
\Vert  \langle x\rangle^{n}e^{it \Delta^2}f\Vert_{L^2(\mathbb{R}^N)}   \lesssim  \langle t\rangle^{n} (\Vert  J^{3n}f\Vert_{L^2(\mathbb{R}^N)}+\Vert  \langle x\rangle^{n}f\Vert_{L^2(\mathbb{R}^N)})  ,
\end{equation}
where the implicit constant above is independent of $t$. 
\end{lemma}
The previous estimate is obtained through a frequency analysis of the operators $\{e^{it\Delta^2}\}$. By converting weighted estimates into frequency derivative estimates, we derive the desired bound for integer weights of arbitrary order (see Proposition \ref{Deriv} and the interpolation result in Proposition \ref{interpolation}). We also note that the corresponding estimate for fractional weights $0<n<1$ (with a different time bound) is a particular case of the result established in \cite[Theorem 1.1]{Pastor2023}, which is based on the ideas developed in \cite{Ponce2009}.
\smallskip

In Theorem \ref{theorem1}, we establish local well-posedness in a subspace of $H^M(\mathbb{R}^N)$. Recall that a Cauchy problem is said to be \emph{unconditionally well-posed} if one has existence, uniqueness, and continuous dependence of the solution map that hold directly in the space $H^s(\mathbb{R}^N)$ (see \cite[Definition~3.4]{Tao2006}). In this direction, using standard techniques that do not rely on dispersive effects, we establish the following result.  

\begin{theorem}\label{LWPresults}
Let $N\in \mathbb{Z}^{+}$, $s>\frac{N}{2}$, $\alpha \geq \lceil s \rceil$\footnote{Recall that the ceiling function $\lceil x \rceil$ is defined as the smallest integer greater than or equal to $x$.}, $\lambda\in \mathbb{C}$, and $a\in \mathbb{R}$. Then the Cauchy problem associated to \eqref{biNLS} is unconditionally well-posed in $H^s(\mathbb{R}^N)$.     
\end{theorem}

The novelty of the above theorem lies in its application to fractional indexes $s> \frac{N}{2}$, which is particularly delicate due to the nonlinear term $|u|^{\alpha}u$ (see Lemma \ref{nonlinearHS} below). In its proof, we use the fractional derivative from Theorem \ref{stein} below and Lemmas \ref{lemmafracderNON} and \ref{GlemmafracderNON}, which are the key for having an explicit control of several nonlinear terms; this can also be useful for other dispersive models. 

\smallskip

The paper is organized as follows: in Section \ref{S:2} we review the notation and discuss the estimates for the bi-laplacian group $e^{it \Delta^2}$, in particular, weighted estimates in Lemma \ref{linearweightedspaces}.
In the following Section \ref{S:3} we state the linear and nonlinear estimates and prove Theorem \ref{theorem1} in Section \ref{S:proof}. We conclude this paper with 
the $H^s$ well-posedness in Section \ref{S:Hs}, including a key result  
for estimating fractional derivatives of nonlinearities Lemmas \ref{fracderivalphaD} and \ref{GlemmafracderNON}.

\subsection*{Notation} Given two positive quantities $A$ and $B$, $A\lesssim B$ means there is a constant $C>0$ such that $A\leq CB$. The notation $A\sim B$ means $A\lesssim B$ and $B\lesssim A$. The Fourier transform of a complex-valued function is denoted by $\widehat{f}$, which is defined as
$\widehat{f}(\xi)=\int_{\mathbb{R}^N}f(x) e^{-i \, x \cdot \xi}\, dx$.  
We denote by $f^{\vee}$ the inverse Fourier transform. Given an integer $k\geq 0$, $\{e^{it (-\Delta)^k}\}$ denotes the group of solutions associated with the linear equation $i\partial_t v+(-\Delta)^k v=0$. The nonhomogeneous fractional derivative of order $b\in\mathbb{R}$ is given by $J^bf(x)=(\langle \xi \rangle^b \hat{f} )^{\vee} (x)$. The homogeneous fractional derivative of order $b\in \mathbb{R}$ is defined via the Fourier transform by $D^bf(x)=(-\Delta)^{\frac{b}{2}}f(x)=( |\xi|^b\hat{f})^{\vee}(x)$. The space $H^s(\mathbb{R}^N)$ is the $L^2$-based Sobolev space with regularity $s\in \mathbb{R}$, equipped with the norm
$
\|f\|_{H^s} \equiv \|J^s f\|=\|\langle \xi \rangle^s \widehat{f}(\xi)\|_{L^2}.  
$

\subsection*{Acknowledgments} The research of all three authors was partially funded by the NSF grant DMS-2055130 (PI: S. Roudenko). Research of I.P. and S.R. was also partially supported by the NSF grant DMS-2452782.

\smallskip

\section{Preliminary estimates}
\label{S:2}

\subsection{Weighted estimates}
We, first, study the derivatives of the operator $e^{it|\xi|^4}$. This is a necessary step for the weighted estimates for the linear group $e^{it\Delta^2}$, which we obtain afterward. We note that the estimates in this section are valid for any dimension $N$.  

\begin{prop}\label{Deriv}
Let $\beta$ be a multi-index, $t \in \mathbb{R}$, and given an integer $l\geq 0$, $P_l(\xi)$ denotes a homogeneous polynomial of order $l$ with variables $\xi\in \R^N$. Then  
\begin{align}\label{derivative}
   \partial^{\beta}(e^{it|\xi|^4})= e^{it|\xi|^4}
    \sum_{l=0}^{\lfloor3|\beta|/4\rfloor} t^{|\beta| -l}P_{3|\beta|-4l}(\xi).
\end{align}
\end{prop}

\begin{proof}
The proof follows by induction on the magnitude of the multi-index $|\beta|\geq 1$. We consider several cases. 

\underline{Case $|\beta|=1$}. We have $\partial^{\beta}=\partial_{\xi_j}$, for some $j=1,\dots N$. A direct computation reveals 
\begin{equation*}
\partial_{\xi_j}(e^{it|\xi|^4})=4it|\xi|^2\xi_j e^{it|\xi|^4},
\end{equation*}
which is of the form \eqref{derivative}, when $|\beta|=1$.

\medskip
By inductive hypothesis, we assume that \eqref{derivative} is valid for any multi-index $\beta$ of order $k\geq 1$.  Thus, we proceed to study the case $|\beta|=k+1$.

\medskip
\underline{Case $|\beta|=k+1$}. We write $\partial^{\beta}=\partial_{\xi_j}\partial^{\beta_1}$ for some $j=1,\dots,N$. Given that $|\beta_1|\geq 1$, we invoke the induction hypothesis to get
\begin{align}\label{beta1}
\begin{split}
    \partial_{\xi}^{\beta}(e^{it|\xi|^4}) &= \partial_{\xi_j}\Big(e^{it|\xi|^4}\sum_{l=0}^{\lfloor3|\beta_1|/4\rfloor} t^{|\beta_1| -l}P_{3|\beta_1|-4l}(\xi)\Big)\\
    &= \partial_{\xi_j}\big(e^{it|\xi|^4}\big)\sum_{l=0}^{\lfloor3|\beta_1|/4\rfloor} t^{|\beta_1| -l}P_{3|\beta_1|-4l}(\xi)+e^{it|\xi|^4}\partial_{\xi_j}\Big(\sum_{l=0}^{\lfloor3|\beta_1|/4\rfloor} t^{|\beta_1| -l}P_{3|\beta_1|-4l}(\xi)\Big)\\
    &= e^{it|\xi|^4}\sum_{l=0}^{\lfloor3|\beta_1|/4\rfloor} t^{|\beta_1|+1 -l}\big(4i|\xi|^2\xi_jP_{3|\beta_1|-4l}(\xi)\big)+e^{it|\xi|^4}\partial_{\xi_j}\Big(\sum_{l=0}^{\lfloor3|\beta_1|/4\rfloor} t^{|\beta_1| -l}P_{3|\beta_1|-4l}(\xi)\Big)\\
    &=:\mathcal{A}_1+\mathcal{A}_2.
\end{split}
\end{align}
Noticing that $|\xi|^2\xi_jP_{3|\beta_1|-4l}(\xi)$ determines a homogeneous polynomial of order $3|\beta_1|+3-4l=3|\beta|-4l$ (recall that $|\beta|=1+|\beta_1|$), we set $Q_{3|\beta|-4l}(\xi) : =  4 i |\xi|^2 \xi_j P_{3|\beta_1|-4l}(\xi)$ to get
\begin{equation*}
\mathcal{A}_1=e^{it|\xi|^4}\sum_{l=0}^{\lfloor3(|\beta|-1)/4\rfloor} t^{|\beta| -l}Q_{3|\beta|-4l}(\xi),   
\end{equation*}
which is part of the summation in \eqref{derivative}. 

On the other hand, to estimate $\mathcal{A}_2$, we further divide our arguments into the cases: $|\beta_1|=4m$ for some $m\in \mathbb{Z}^{+}\cup\{0\}$, and $|\beta_1|=4m+r$ for some $m\in \mathbb{Z}^{+}\cup\{0\}$ and $r=1,2,3$. 

\medskip
\underline{Case $|\beta_1|=4m$ for some $m\in \mathbb{Z}^{+}\cup\{0\}$}. We note that
\begin{equation*}
 \lfloor 3|\beta_1|/4 \rfloor=3m=\lfloor 3|\beta|/4 \rfloor.
\end{equation*}
Since $3|\beta_1|-4l-1=3(|\beta_1|+1)-4(l+1)=3|\beta|-4(l+1)$, we have that $Q_{3|\beta|-4(l+1)}(\xi):=\partial_{\xi_j}P_{3|\beta_1|-4l}(\xi)$ is a homogeneous polynomial of degree $3|\beta|-4(l+1)$, whenever $l=0,\dots, \lfloor3|\beta_1|/4\rfloor-1$. Observe that if $l=3m$, $P_{3|\beta_1|-4l}(\xi)$ is a non-null constant polynomial. When differentiated, we omit it from the summation. Hence, we arrive at
\begin{equation*}
\begin{aligned}
 \mathcal{A}_2=e^{it|\xi|^4}\sum_{l=0}^{\lfloor3|\beta_1|/4\rfloor-1} t^{|\beta_1| -l} Q_{3|\beta|-4(l+1)}(\xi)=e^{it|\xi|^4}\sum_{l=1}^{\lfloor3|\beta|/4\rfloor} t^{|\beta| -l} Q_{3|\beta|-4l}(\xi).
\end{aligned}    
\end{equation*}
The above expression completes the deduction of \eqref{derivative} in this case.
\smallskip

\underline{Case $|\beta_1|=4m+r$ for some $m\in \mathbb{Z}^{+}\cup\{0\}$, $r=1,2,3$}. We have
\begin{equation*}
\lfloor 3|\beta|/4 \rfloor=\lfloor 3|\beta_1|/4 \rfloor+1.   
\end{equation*}
As a consequence, $Q_{3|\beta|-4(l+1)}(\xi):=\partial_{\xi_j}P_{3|\beta_1|-4l}(\xi)$ is a non-null homogeneous polynomial of degree $3|\beta|-4(l+1)$ for all $l=0,\dots, \lfloor3|\beta_1|/4\rfloor$. Hence, the argument in the previous case shows
\begin{equation*}
\begin{aligned}
 \mathcal{A}_2=e^{it|\xi|^4}\sum_{l=0}^{\lfloor3|\beta_1|/4\rfloor} t^{|\beta_1|} Q_{3|\beta|-4(l+1)}(\xi)=e^{it|\xi|^4}\sum_{l=1}^{\lfloor3|\beta_1|/4\rfloor+1} t^{|\beta| -l} Q_{3|\beta|-4l}(\xi).
\end{aligned}    
\end{equation*}
The above expression establishes that $\mathcal{A}_2$ is part of the summation in \eqref{derivative}.

Collecting the identities obtained for $\mathcal{A}_1$ and $\mathcal{A}_2$, we complete the induction step $|\beta|=k+1$. By the induction principle, the desired result follows.
\end{proof}

We next would like to prove a weighted estimate of the linear flow $e^{it\Delta^2} f$, i.e., Lemma \ref{linearweightedspaces},  though prior that we state a useful interpolation result, which we borrow from  \cite{Ponce2009}.
\begin{prop}(\cite[Lemma 4]{Ponce2009})\label{interpolation}
For any $a,b>0$, $x \in \mathbb R^N$, and $\theta\in (0,1)$, we have 
\begin{equation}
   \Vert  \langle x\rangle^{\theta a}(J^{(1-\theta)b}f)\Vert_{L^2}   \lesssim  \Vert J^b f\Vert_{L^2}^{1-\theta} \Vert  \langle x\rangle^{\theta }f\Vert_{L^2}^{\theta},
\end{equation}
\begin{equation}
   \Vert  J^{\theta a}(\langle x\rangle^{(1-\theta)b}f)\Vert_{L^2}   \lesssim  \Vert \langle x\rangle^b f\Vert_{L^2}^{1-\theta} \Vert  J^{a }f\Vert_{L^2}^{\theta}.
\end{equation}
\end{prop}

We are now ready to prove Lemma \ref{linearweightedspaces}.

\begin{proof}[Proof of Lemma \ref{linearweightedspaces}] By Plancherel's identity and Leibniz's rule,
\begin{align}\label{weighted-flow-start}
\begin{split}
\Vert  \langle x\rangle^{n}e^{it \Delta^2} f \Vert_{L^2} =&  \Vert  J_{\xi}^{n}(e^{it|\xi|^4}f)\Vert_{L^2} \\ 
  \lesssim & \Vert  f\Vert_{L^2} +\sum_{j=1}^{N} \Vert \partial_{{\xi}_j}^{n}(e^{it|\xi|^4}\hat{f})\Vert_{L^2}\\
  \lesssim & \Vert  f\Vert_{L^2} +\sum_{j=1}^{N}\sum_{k=0}^{n}\Vert  \partial_{{\xi}_j}^{k}(e^{it|\xi|^4})\, \partial_{{\xi}_j}^{n-k}\hat{f}\,
  \Vert_{L^2}.
\end{split}
\end{align}
To estimate the previous inequality, we apply Proposition \ref{Deriv} to get
\begin{equation*}
\begin{aligned}
\sum_{j=1}^{N}\sum_{k=0}^{n}\Vert  \partial_{{\xi}_j}^{k}(e^{it|\xi|^4})\partial_{{\xi}_j}^{n-k}\hat{f} \Vert_{L^2} \lesssim & \sum_{j=1}^{N} \sum_{k=0}^{n}\sum_{l=0}^{\lfloor3k/4\rfloor} |t|^{k -l}\|P_{3k-4l}(\xi)\,\partial_{{\xi}_j}^{n-k}\hat{f}
  \|_{L^2}\\ 
  \lesssim & \langle t\rangle^{n}\sum_{j=1}^{N}\sum_{k=0}^{n}\sum_{l=0}^{\lfloor3k/4\rfloor}\|\langle\xi \rangle^{3k 4l}\,\partial_{{\xi}_j}^{n-k}\hat{f}  \|_{L^2} .   
\end{aligned} 
\end{equation*}
Next, using that $\|\langle\xi \rangle^{3k-4l}\,\partial_{{\xi}_j}^{n-k}\hat{f}
  \|_{L^2} \lesssim \|J^{n-k}\big(\langle\xi \rangle^{3k-4l}\,\hat{f}\big)
  \|_{L^2}$, whenever $k\neq 0$, an application of Proposition \ref{interpolation} with $\theta=\frac{n-k}{n}$ yields
  \begin{equation*}
   \begin{aligned}
   \|J^{n-k}\big(\langle\xi \rangle^{3k-4l}\,\hat{f}\big)
  \|_{L^2}\lesssim & \|\langle\xi \rangle^{\big(\frac{3k-4l}{k}\big)n}\,\widehat{f}\|_{L^2}^{1-\theta}\,\|J^n\widehat{f}\|^{\theta}_{L^2}\\
  \lesssim & \|\langle\xi \rangle^{3n}\,\widehat{f}\|_{L^2}+\|J^n\widehat{f}\|_{L^2},
   \end{aligned}   
  \end{equation*}
 where we have used that $\big(\frac{3k-4l}{k}\big)n\leq 3n$ and Young's inequality. Going back to \eqref{weighted-flow-start}, Plancherel's identity yields
 \begin{align}\label{weighted-flow-final}
\begin{split}
 \Vert  \langle x\rangle^{n}e^{it\Delta^2}f\Vert_{L^2} \lesssim& \langle t\rangle^{n}(\|\langle\xi \rangle^{3n}\,\widehat{f}\|_{L^2}+\|J^n\widehat{f}\|_{L^2}) \\
 \lesssim& \langle t\rangle^{n}(\|J^{3n}f\|_{L^2}+\|\langle x\rangle^{n}f\|_{L^2}).
\end{split}
\end{align}
This completes the proof of \eqref{weightedineq-biharmonic} and finishes the lemma.
\end{proof} 
Finally, we recall an estimate for the linear equation associated with the standard linear Schr\"odinger equation $i\partial_t u +\Delta u=0$. For its proof, we refer the reader to \cite{Ponce2009,ARR2022}.

\begin{lemma}\label{Lineareqweight}
Let $N\in \mathbb{Z}^{+}$, $n\in \mathbb{R}^{+}$, then for any $t\in \mathbb{R}$ 
\begin{equation}\label{weightedineq-schrodinger}
\Vert  \langle x\rangle^{n}e^{it \Delta} f \Vert_{L^2(\mathbb{R}^N)}   \lesssim  \langle t\rangle^{n} (\Vert  J^{n}f\Vert_{L^2(\mathbb{R}^N)}+\Vert  \langle x\rangle^{n}f\Vert_{L^2(\mathbb{R}^N)}).
\end{equation}
\end{lemma}

\section{Local well-posedness in weighted space \texorpdfstring{$\mathcal X$}{X}}
\label{S:3}

\subsection{Linear estimates}

In this part, we obtain key linear and nonlinear estimates for our local well-posedness result, which we then use to prove Theorem \ref{theorem1}. 
\begin{lemma}\label{pure bi-laplacian}
Let $N\in\{1,2,3\}$, $a=0$. Consider $\psi \in \X_{N}$, where the space $\X_{1}$ is defined in \eqref{spaceX1-1D} and for $N=2,3$, the space $\X_N$\ is given by \eqref{spaceX1 2D}. Then, for all $t\in \R$ 
\begin{equation}
    \Vert e^{it\Delta^2}\psi\Vert_{\X_N}\lesssim\langle t\rangle^{n+1}\Vert\psi\Vert_{\X_N}
\end{equation}
and
\begin{equation}\label{diffsol1}
    \Vert \langle x\rangle^n(e^{it\Delta^2}\psi-\psi)\Vert_{L^{\infty}}\lesssim |t|\langle t\rangle^{n}\Vert\psi\Vert_{\X_N},
\end{equation}
where the implicit constants above are independent of $t$.
\end{lemma}

\begin{lemma}\label{linear est l-laplacian}
Let $N\in\{1,2,3\}$, $a\in \mathbb{R} \setminus \{0\}$. Consider $\psi \in \X_{N}$, where the space $\X_{1}$ is defined as in \eqref{spaceX1-1D} and for $N=2,3$, the space $\X_N$\ is given by \eqref{spaceX1 2D}. Then, for all $t\in \R$ 
\begin{equation}
    \Vert e^{it(\Delta^2+a\Delta)}\psi\Vert_{\X_N}\lesssim\langle t\rangle^{2n+1}\Vert\psi\Vert_{\X_N}
\end{equation}
and
\begin{equation}\label{diffsol2}
    \Vert \langle x\rangle^n(e^{it(\Delta^2+a\Delta)}\psi-\psi)\Vert_{L^{\infty}}\lesssim |t|\langle t\rangle^{2n}\Vert\psi\Vert_{\X_N},
\end{equation}
where the implicit constants above are independent of $t$.
\end{lemma}

We only prove Lemma \ref{linear est l-laplacian}, since Lemma \ref{pure bi-laplacian} follows similar ideas, 
see also remarks in the proof for that case.

\begin{proof}[Proof of Lemma \ref{linear est l-laplacian}]
To simplify notation, we assume $t\geq 0$. The case $t<0$ follows the same arguments. 
We apply the Fundamental Theorem of Calculus to the linear flow of \eqref{biNLS} to get

\begin{equation}\label{FTC}
 e^{it(\Delta^2+2a\Delta)}\psi=\psi + i\int_0^{t}e^{is(\Delta^2+2a\Delta)} (\Delta^2+2a\Delta)\psi(s) \, ds. 
\end{equation}
Applying a derivative of order $\beta$ on both sides of \eqref{FTC} yields
\begin{equation}
\partial^{\beta}e^{it(\Delta^2+2a\Delta)}\psi=\partial^{\beta}\psi+i\int_0^{t}e^{is(\Delta^2+2a\Delta)} (\Delta^2+2a\Delta)\partial^{\beta}\psi(s) \,ds.   
\end{equation}

We next compute each norm in the definition of the space $\X_{N}$ for $N=1,2,3$ in \eqref{spaceX1-1D} and \eqref{spaceX1 2D}.
\smallskip

$\bullet$ \underline{\bf $L^{\infty}$-norm}.
For $\beta=0$, we have     \begin{align}
\Vert \langle x\rangle^ne^{it(\Delta^2+2a\Delta)}\psi\Vert_{L^{\infty}}&\lesssim\Vert \langle x\rangle^n\psi\Vert_{L^{\infty}}+\int_0^{t}\Vert \langle x\rangle^n e^{is(\Delta^2+2a\Delta)} (\Delta^2+2a\Delta)\psi (s)\Vert_{L^{\infty}} \, ds  \notag \\
&\lesssim \Vert \psi\Vert_{\X_{N}}+\int_0^{t}\Vert \langle x\rangle^n e^{is(\Delta^2+2a\Delta)}(\Delta^2+2a\Delta)\psi (s) \Vert_{L^{\infty}}\, ds. \label{linear-inf}   
\end{align}
Define $m_N$ by setting $m_1=1$, $m_2=m_3=2$.
To control the previous inequality, we apply Sobolev embedding $H^{m_N}(\mathbb{R}^{N})\hookrightarrow L^{\infty}(\mathbb{R}^N)$, whenever $N=1,2,3$. We also use a trivial estimate  $|\partial^{\alpha}(\langle x\rangle^{n})|\lesssim \langle x\rangle^{n}$ for any multi-index $\alpha$, and we bound the integrand in \eqref{linear-inf} as
\begin{equation*}
  \begin{split} 
      \Vert \langle x\rangle^n e^{is(\Delta^2+2a\Delta)}(\Delta^2+2a\Delta)\psi \Vert_{L^{\infty}}&\lesssim\sum_{|\sigma|=0}^{m_N}\Vert \partial^{\sigma}\Big(\langle x\rangle^n e^{is(\Delta^2+2a\Delta)}(\Delta^2+2a\Delta)\psi\Big)\Vert_{L^{2}}\\
      \lesssim \sum_{|\sigma|=4}^{m_N+4}\Vert \langle x\rangle^n &e^{is(\Delta^2+2a\Delta)}\partial^{\sigma}\psi \Vert_{L^2} + 2|a| \sum_{|\sigma|=2}^{m_N+2}\Vert \langle x\rangle^n e^{is(\Delta^2+2a\Delta)}\partial^{\sigma}\psi\Vert_{L^2}.
      \end{split}
  \end{equation*}  
Since $a \neq 0$, we rewrite $e^{is(\Delta^2+2a\Delta)}=e^{is \Delta^2}\cdot e^{2ias\Delta}$, $0<s\leq t$, and apply Lemma \ref{linearweightedspaces} for the term with the double Laplacian and  Lemma \ref{Lineareqweight} for the term with the single Laplacian to deduce
\begin{align*}
\Vert \langle x\rangle^n e^{is(\Delta^2+2a\Delta)}(\Delta^2+2a\Delta)\psi \Vert_{L^{\infty}}&\lesssim\langle s\rangle^n\sum_{|\sigma|=2}^{m_N+4}(\Vert J^{3n}(e^{2ias\Delta}\partial^{\sigma}\psi)\Vert_{L^2}+\Vert \langle x\rangle^n (e^{2ias\Delta}\partial^{\sigma}\psi)\Vert_{L^2})\\
&{ \lesssim\langle s\rangle^n\sum_{|\sigma|=2}^{m_N+4}\Big(\Vert J^{3n}\partial^{\sigma}\psi\Vert_{L^2}+\langle s \rangle^n(\Vert J^{n} \partial^{\sigma}\psi\Vert_{L^2}+\Vert \langle x\rangle^n \partial^{\sigma}\psi\Vert_{L^2})\Big)}\\
&{ \lesssim \langle s\rangle^{2n}\Big( \sum_{|\sigma|=2}^{3n+m_N+4}\Vert \partial^{\sigma}\psi\Vert_{L^2}+\sum_{|\sigma|=2}^{m_N+4}\Vert \langle x\rangle^n \partial^{\sigma}\psi\Vert_{L^2}\Big)}\\
&{ \lesssim \langle s\rangle^{2n}\Big( \sum_{|\sigma|=2}^{r}\Vert \langle x \rangle^n\partial^{\sigma}\psi\Vert_{L^2}+\sum_{|\sigma|=r}^{M}\Vert \partial^{\sigma}\psi\Vert_{L^2}\Big)}\\
    &\lesssim\langle s\rangle^{2n}\Vert \psi\Vert_{\X_{N}},
\end{align*}
{  where we have also used that $\|\partial^{\sigma}\psi\|_{L^2}\lesssim \|\langle x\rangle^n\partial^{\sigma}\psi\|_{L^2}$, whenever $2 \leq |\sigma|\leq r$}. Note that since $a \neq 0$, we had to apply Lemmas \ref{linearweightedspaces} and \ref{Lineareqweight}, which double the power of $\langle s \rangle$, hence, the power is $2n$ in the estimates of Lemma \ref{linear est l-laplacian} and only $n$ power in those of Lemma \ref{pure bi-laplacian}.  

Collecting the previous estimates, we obtain
\begin{equation}
\Vert \langle x\rangle^ne^{it(\Delta^2+2a\Delta)}\psi\Vert_{L^{\infty}} \lesssim \Vert\psi\Vert_{\X_{N}}+\langle t\rangle^{2n}\int_0^{t}\Vert  \psi\Vert_{\X_{N}} ds\lesssim\langle t\rangle^{2n+1}\Vert \psi\Vert_{\X_{N}}.
\end{equation}
In particular, the argument above establishes \eqref{diffsol2}. {  Observe that the previous estimate requires $r\geq m_N+4$}. For $|\beta|=1$ (which happens for dimension $N=2,3$), an analogous argument yields
\begin{equation*}
\Vert \langle x\rangle^n\partial^{\beta}\big( e^{it(\Delta^2+2a\Delta)}\psi\big)\Vert_{L^{\infty}}\lesssim\langle t\rangle^{2n+1}\Vert \psi\Vert_{\X_{N}}.
\end{equation*}
{  Note that, in view of the previous arguments, the above estimates require $r\geq m_N+5=7$, $N=2,3$}.

$\bullet$ \underline{\bf Weighted $L^{2}$-norm}. If $N=1$, consider $1\leq|\beta|\leq r$, and if $N=2,3$, take $2\leq|\beta|\leq r$. Applying Lemmas \ref{linearweightedspaces} and \ref{Lineareqweight}, and recalling that $M=3n+r$, we get
 { \begin{align*}
     \begin{split}
          \Vert\langle x\rangle^n e^{it(\Delta^2+2a\Delta)}\partial^{\beta}\psi\Vert_{L^{2}}&\lesssim  \langle t \rangle^{n}(\Vert J^{3n}\partial^{\beta}\psi\Vert_{L^2}+\Vert \langle x\rangle^n e^{2ia t\Delta}\partial^{\beta}\psi\Vert_{L^2})\\
&\lesssim \langle t \rangle^{n}(\sum_{|\beta|\leq |\sigma|\leq r}\Vert \langle x \rangle^n\partial^{\sigma}\psi\Vert_{L^2}+\sum_{r+1\leq |\sigma|\leq M}\Vert\partial^{\sigma}\psi\Vert_{L^2}+\Vert \langle x\rangle^n e^{2ia t\Delta}\partial^{\beta}\psi\Vert_{L^2})\\         &\lesssim \langle t \rangle^{n} \Big( \Vert \psi \Vert_{\X} +\langle t \rangle^{n}(\Vert J^{n}\partial^{\beta}\psi\Vert_{L^2}+\Vert \langle x\rangle^n \partial^{\beta}\psi\Vert_{L^2}) \Big)\\
& \lesssim \langle t \rangle^{2n}\Vert \psi\Vert_{\X_N}.
 \end{split}
\end{align*}
}
\smallskip

$\bullet$ \underline{\bf Unweighted $L^2$-norm}. Let $r+1\leq |\beta|\leq M$. Using the fact that $\{e^{it(\Delta^2+2a\Delta)}\}_{t\in \mathbb{R}}$ is a group of isometries in $H^s$-spaces, we get 
\begin{equation*}
\Vert \partial^{\beta}\big( e^{it\Delta^2+2a\Delta)}\psi\big)\Vert_{L^2}=\Vert  \partial^{\beta} \psi\Vert_{L^2}\lesssim \Vert \psi\Vert_{\X_{N}}.
\end{equation*}
Collecting all above estimates that compose the $\mathcal X_N$ norm yields the desired result. 
\end{proof}

\subsection{Nonlinear estimate}

In this part, we deduce nonlinear estimates for $|u|^{\alpha }u$ and $|u|^{\alpha }u-|v|^{\alpha }v$ in the space $\X_N$, $N=1,2,3$, whenever $u,v \in \X_N$ and satisfy a non-vanishing condition \eqref{inf cond u0}. The main result is the following:

\begin{lemma} \label{nonlinear est}
Let $\alpha > 0$ and $a \in \mathbb{R}$. For dimension $N = 1$, let $n, M, r$ be defined by \eqref{E:n-1d} and \eqref{SpaceDim1gen}, and for dimensions $N = 2, 3$, by \eqref{E:n-2D-3D} and \eqref{spaceX1 gen 2D}. Consequently, define the space $\X_N$ by \eqref{spaceX1-1D} for $N=1$, and by \eqref{spaceX1 2D} for $N=2,3$. Suppose that $u \in \X_N$ satisfies 
\begin{equation}\label{inf cond sol}
       \inf_{x \in \mathbb{R}^N}\left(\langle x \rangle^{n}|u(x)|\right) \geq \frac{1}{\eta} \quad \mbox{for ~~some~~} \eta>0.
\end{equation} 
Then $|u|^{\alpha}u\in \X_N$. Moreover, 
   \begin{equation}\label{nonlinear}
   \Vert |u|^{\alpha}u \Vert_{\X_N} \lesssim (1+\eta\Vert u\Vert_{\X_N})^{2M}\Vert u\Vert ^{\alpha + 1}_{\X_N}.
   \end{equation}  
Additionally, if $u,v\in \X_N$ and satisfy \eqref{inf cond sol}, then
\begin{equation}\label{nonlinear diffrerence}
\begin{aligned}  
       \Vert|u|^{\alpha}u-|v|^{\alpha}v\Vert_{\X_N} &\lesssim \mathcal{G}(\eta,\|u\|_{\X_N},\|v\|_{\X_N}) \Vert u-v\Vert_{\X_N},   
\end{aligned}
\end{equation}
where    
    \begin{equation}\label{nonlinearGfunc}
    \begin{aligned}  
    \mathcal{G}(\eta,\|u\|_{\X_N},\|v\|_{\X_N})=&\sum_{p=0}^M\Big(\big(\eta+\Vert u \Vert_{\X_N}+\Vert v\Vert_{\X_N}\big)^{|\alpha-2p-1|+2p+1}+\big(\eta+\Vert u \Vert_{\X_N}+\Vert v\Vert_{\X_N}\big)^{4p+\alpha}\Big).  
    \end{aligned}
    \end{equation}
\end{lemma}
\begin{proof}
We first estimate $\partial^{\beta}(|u|^{\alpha}u)$, $1\leq |\beta |\leq M$. Applying Leibniz's rule, we get
\begin{align}\label{deriv}
\partial^{\beta}(|u|^{\alpha}u)
   & =|u|^{\alpha}\partial^{\beta} u + \sum^{|\beta|}_{|\gamma|=1}C_{\gamma}\partial^{\gamma} (|u|^{\alpha})\partial^{\beta-\gamma}u =:\mathscrsfs{A}+\mathscrsfs{B}.
   \end{align}
Writing $|u|^{\alpha}=(u\bar{u})^{\alpha/2}$, we obtain
\begin{equation}\label{B}
\mathscrsfs{B}=  
\sum^{|\beta|}_{|\gamma|=1}\sum^{|\gamma|}_{p=1}\sum_{\substack{\gamma_1+\dots+\gamma_p=\gamma,\\|\gamma_j|\geq1}}\sum_{\substack{\gamma_{1,j}+\gamma_{2,j}=\gamma_j,\\1\leq j\leq p}} c_{\beta,\gamma_{1,1},\dots,\gamma_{2,p}}|u|^{\alpha-2p}\,\,\partial^{\beta-\gamma} u \,\,\partial^{\gamma_{1,1}} u\,\,\partial^{\gamma_{2,1}}\bar{u}\dots\partial^{\gamma_{1,p}} u\,\,\partial^{\gamma_{2,p}} \bar{u},
    \end{equation}
for some constants $c_{\beta,\gamma_{1,1},\dots,\gamma_{2,p}}$. 
We are now ready to estimate each of the three parts of the $\X_N$-space norm: weighted and unweighted $L^{\infty}$-norm, weighted $L^2$-norm, and $H^M$-norm.
\smallskip

$\bullet$ \textbf{\underline{Weighted $L^{\infty}$-norm}}. 
Let a multi-index $\beta$ be such that $0\leq|\beta|\leq 1$ when the dimensions $N=2$ and $3$, and take $\beta=0$ in the dimension $N=1$ (i.e., we consider the function itself in dimensions $N=1,2,3$ and its first derivative in dimensions $N=2,3$). Then 
\begin{equation}
\Vert \langle x\rangle^{n}|u|^{\alpha}u\Vert_{L^{\infty}}=\Vert \langle x\rangle^{n\alpha} |u|^{\alpha}\langle x\rangle^{-n\alpha}\langle x\rangle^{n}  u \Vert_{L^{\infty}}=\Vert \langle x\rangle^{-n\alpha}\Vert_{L^{\infty}}\Vert\langle x\rangle^{n} u \Vert^{\alpha}_{L^{\infty}}\Vert \langle x\rangle^{n}  u \Vert_{L^{\infty}}
\leq  \Vert u\Vert_{\X_N}^{\alpha+1},
\end{equation}
and for $j=1,..,N$, we also bound one derivative $\partial_{x_j}$ with a weight as
\begin{align*} 
\Vert \langle x\rangle^{n}\partial_{x_j}(|u|^{\alpha}u)\Vert_{L^{\infty}}&=\Vert\langle x\rangle^{n}(\partial_{x_j}(|u|^{\alpha})u
+|u|^{\alpha}\partial_{x_j} u)\Vert_{L^{\infty}}
\\
&   \lesssim \Vert \langle x\rangle^{n}|u|^{\alpha}\partial_{x_j} u\Vert_{L^{\infty}}
\leq \Vert \langle x\rangle^{-n\alpha}\Vert_{L^{\infty}}\Vert\langle x\rangle^{n} u \Vert^{\alpha}_{L^{\infty}}\Vert \langle x\rangle^{n}  \partial_{x_j} u \Vert_{L^{\infty}}
\leq \Vert u\Vert_{\X_N}^{\alpha+1}.
\end{align*}

$\bullet$ \textbf{\underline{Weighted $L^{2}$-norm}}. Consider $1\leq |\beta|\leq r$ for $N=1$, and  $2\leq |\beta|\leq r$, when $N=2,3$. By an application of H\"older's inequality, we obtain
\begin{align*}
\Vert \langle x\rangle^n \mathscrsfs{A} \Vert_{L^{2}}
=\Vert \langle x\rangle^{n\alpha} |u|^{\alpha}\langle x\rangle^{-n\alpha}\langle x\rangle^{n} \partial^{\beta} u \Vert_{L^{2}}
\leq \Vert \langle x\rangle^n u \Vert_{L^{\infty}} ^{\alpha}\Vert \langle x\rangle^{n} \partial^{\beta} u \Vert_{L^{2}}
\leq \Vert u\Vert_{\X_N}^{\alpha+1}.
\end{align*}
Before we estimate the weighted term with $\mathscrsfs{B}$, we recall the infimum condition \eqref{inf cond sol} and apply it in the following manner
\begin{equation}\label{neg power}
|u|^{\alpha-2p}=|u|^{\alpha}|u|^{-2p}\leq\eta^{2p}\langle x\rangle^{2pn}|u|^{\alpha}.
\end{equation}
Then 
\begin{align}
\Vert &\langle x\rangle^{n}\mathscrsfs{B}\Vert_{L^2} \notag \\
&\lesssim\sum^{|\beta|}_{|\gamma|=1}\sum^{|\gamma|}_{p=1}\sum_{\substack{\gamma_1+\dots+\gamma_p=\gamma,\\|\gamma_j|\geq1}}\sum_{\substack{\gamma_{1,j}+\gamma_{2,j}=\gamma_j,\\1\leq j\leq p}}  \Vert \langle x\rangle^n|u|^{\alpha-2p}\,\,\partial^{\beta-\gamma} u \,\,\partial^{\gamma_{1,1}} u\,\,\partial^{\gamma_{2,1}}\bar{u}\dots\partial^{\gamma_{1,p}} u\,\,\partial^{\gamma_{2,p}} \bar{u}\Vert_{L^{2}} \notag \\
&\lesssim\sum^{|\beta|}_{|\gamma|=1}\sum^{|\gamma|}_{p=1}\sum_{\substack{\gamma_1+\dots+\gamma_p=\gamma,\\|\gamma_j|\geq1}}\sum_{\substack{\gamma_{1,j}+\gamma_{2,j}=\gamma_j,\\1\leq j\leq p}}\Vert \langle x\rangle^{n\alpha} |u|^{\alpha}\Vert_{L^{\infty}}\Vert \eta ^{2p} \langle x\rangle^{(-\alpha+2p+1)n}\partial^{\beta-\gamma} u \,\,\partial^{\gamma_{1,1}} u\dots\partial^{\gamma_{2,p}} \bar{u}\Vert_{L^2} \notag \\ &\lesssim\sum^{|\beta|}_{|\gamma|=1}\sum^{|\gamma|}_{p=1}\sum_{\substack{\gamma_1+\dots+\gamma_p=\gamma,\\|\gamma_j|\geq1}}\sum_{\substack{\gamma_{1,j}+\gamma_{2,j}=\gamma_j,\\1\leq j\leq p}} \eta ^{2p}\,\,\Vert u\Vert^{\alpha}_{\X_N}\,
\underbrace{\Vert\langle x\rangle^{-n\alpha}\langle x\rangle^{2pn+n}\,\,\partial^{\beta-\gamma} u \,\,\partial^{\gamma_{1,1}} u\dots\partial^{\gamma_{2,p}} \bar{u}\Vert_{L^2} }_{I_{\mathscrsfs{B}}}, \label{B L2 weight}
\end{align}
and thus, we need to obtain bounds for $I_{\mathscrsfs{B}}$.
For that we consider three cases.
\smallskip

\underline{Case 1}. We assume that all derivatives $\partial^{\beta-\gamma} u$, $\partial^{\gamma_{1,1}} u,\dots,\partial^{\gamma_{2,p}} \bar{u}$ are of order $\leq 1$. Thus, taking the $L^{\infty}$-norms of all terms with derivatives (and distributing the weights) except one term, which we have to put into $L^2$, we find
\begin{align*}
I_{\mathscrsfs{B}}
&\lesssim  \Vert \langle x\rangle^{-n\alpha}\Vert_{L^2}\Vert\langle x\rangle^{2pn+n}\,\,\partial^{\beta-\gamma} u \,\,\partial^{\gamma_{1,1}} u\dots\partial^{\gamma_{2,p}} \bar{u}\Vert_{L^{\infty}}
\\ &
\lesssim \Vert u\Vert^{2p+1}_{\X_N},    
\end{align*}
which requires $\langle x\rangle^{-n\alpha}\in L^2(\R^N)$, i.e., $n>\frac{N}{2\alpha}$, for $N=1,2,3$. This is the condition that we require in dimensions $N=2$ and $3$ in the definition \eqref{E:n-2D-3D} of the space $\mathcal X_N$. 

In the case of $N=1$, 
we do not have to put (all) first derivatives into $L^\infty$, per definition of $\mathcal X_1$, and thus, we bound $I_{\mathscrsfs{B}}$ slightly differently, avoiding the restriction on $n$. Note that since $|\gamma| \geq 1$, there is at least one term with the first derivative in this case, call it $\partial_{x_{j^*}}$, then 

\begin{align*}
I_{\mathscrsfs{B}}
&\lesssim  \Vert \langle x\rangle^{-n\alpha}\Vert_{L^\infty}
\Vert\langle x\rangle^{2pn}\,\underbrace{\partial^{\beta-\gamma} u \,\,\partial^{\gamma_{1,1}} u\dots\partial^{\gamma_{2,p}} \bar{u}}_{\mbox{except}~~\partial_{x_{j^*}}}\Vert_{L^{\infty}} \|\langle x \rangle^n \partial_{x_{j^*}} u\|_{L^2}
\\ &
\lesssim \Vert u\Vert^{2p+1}_{\X_N}.    
\end{align*}

\underline{Case 2}. Let us assume that $|\beta|\leq r-1$ if $N=1$, and $|\beta|\leq r-2$, for $N=2,3$, and there exists a term among $\partial^{\beta-\gamma} u$, $\partial^{\gamma_{1,1}} u$, $\dots$, $\partial^{\gamma_{2,p}} \bar{u}$ whose derivative is of order $>1$. Before continuing with the estimate, we observe that if $N=1$,  Sobolev embedding $H^{1}(\mathbb{R})\hookrightarrow L^{\infty}(\mathbb{R})$ establishes that for all integers $0\leq \widetilde{\gamma} \leq r-1$,
\begin{equation}\label{Sobembeddingineq1}
\begin{aligned}
    \Vert \langle x\rangle^{n} \partial^{\widetilde{\gamma}} u\Vert_{L^{\infty}} \lesssim & \Vert\langle x \rangle^{n}u\Vert_{L^{\infty}}+ \sum_{1 \leq \beta \leq r-1} \Vert\langle x \rangle^{n}\partial^{\beta} u\Vert_{L^{\infty}}\\
    \lesssim & \Vert\langle x \rangle^{n}u\Vert_{L^{\infty}}+ \sum_{1 \leq \beta \leq r} \Vert\langle x \rangle^{n}\partial^{\beta} u\Vert_{L^{2}}\\
    \lesssim & \|u\|_{\X_1}.
\end{aligned}
\end{equation}
When $N=2,3$, taking an arbitrary multi-index $|\widetilde{\gamma}|\leq r-2$, we can apply the embedding $H^{2}(\mathbb{R}^N)\hookrightarrow L^{\infty}(\mathbb{R}^N)$ to get 
\begin{equation}\label{Sobembeddingineq2}
\begin{aligned}
    \Vert \langle x\rangle^{n} \partial^{\widetilde{\gamma}}u\Vert_{L^{\infty}} \lesssim & \sum_{ |\beta|\leq 1}\Vert\langle x \rangle^{n}\partial^{\beta} u\Vert_{L^{\infty}} + \sum_{2 \leq |\beta| \leq r-2} \Vert\langle x \rangle^{n}\partial^{\beta}u\Vert_{L^{\infty}}\\
    \lesssim & \sum_{ |\beta| \leq 1}\Vert\langle x \rangle^{n}\partial^{\beta} u\Vert_{L^{\infty}} + \sum_{2 \leq |\beta| \leq r} \Vert\langle x \rangle^{n}\partial^{\beta}u\Vert_{L^{2}}\\
    \lesssim & \|u\|_{\X_N}.
\end{aligned}
\end{equation}
Returning to our estimate, for the sake of simplicity in the presentation, we assume that $|\gamma_{1,1}|$ is one of the derivatives of the highest order. Hence, an application of H\"older's inequality, \eqref{Sobembeddingineq1} and \eqref{Sobembeddingineq2}  establish
\begin{equation}\label{eqCase2}
    \begin{aligned}
     \Vert\langle x\rangle^{-n\alpha} &\langle x\rangle^{2pn+n}\,\,\partial^{\beta-\gamma} u \,\,\partial^{\gamma_{1,1}} u\dots\partial^{\gamma_{2,p}} \bar{u}\Vert_{L^2}\\
     &\lesssim \Vert \langle x\rangle^{-n\alpha} \Vert_{L^{\infty}}\Vert \langle x\rangle^{2pn+n} \underbrace{\partial^{\gamma_{1,1}} u}_{\text{ highest order}}\partial^{\beta-\gamma} u\,\,\partial^{\gamma_{1,2}} u\dots\partial^{\gamma_{2,p}} \bar{u}\Vert_{L^2} \\
     &\lesssim \Vert \langle x\rangle^n \partial^{\gamma_{1,1}} u\Vert_{L^2}\,\,\underbrace{ 
    \Vert \langle x\rangle^{n}\,\,\partial^{\beta-\gamma} u \Vert_{L^{\infty}}\,\,\Vert\langle x\rangle^{n}\,\,\partial^{\gamma_{1,2}} u\Vert_{L^{\infty}}\cdot \cdot \cdot  \Vert \langle x\rangle^{n}\,\, \partial^{\gamma_{2,p}} \bar{u}\Vert_{L^{\infty}}}_{2p \text{ times}}\\
    &\lesssim \Vert u \Vert_{\X_N}^{2p+1}.
    \end{aligned}
\end{equation}

\underline{Case 3}. We assume that $|\beta|>r-1$ if $N=1$, and $|\beta|> r-2$, when $N=2,3$, and there exists a term among $\partial^{\beta-\gamma} u$, $\partial^{\gamma_{1,1}} u$, $\dots$, $\partial^{\gamma_{2,p}} \bar{u}$ whose derivative is of order $>1$. For notational convenience, we assume again that $\gamma_{1,1}$ corresponds to the derivative with the highest order. 

If $|\gamma_{1,1}|\leq r-1$ when $N=1$, or $|\gamma_{1,1}|\leq r-2$ when $N=2,3$, then we fall back into Case 2 above. On the other hand, suppose $N=1$ and $|\gamma_{1,1}|> r-1$. By the definition of space $\X_1$, it must hold that $|\beta-\gamma|, |\gamma_{l_1,l_2}|\leq r-1$ for each $l_1=1,2$, $l_2=1,\dots, p$ with $(l_1,l_2)\neq (1,1)$. Otherwise, we would deduce $2r-2<r$, i.e., $r<2$, which is a contradiction to the definition of $r$. 

Similarly, for dimensions $N=2,3$, it must follow that $|\beta-\gamma|, |\gamma_{l_1,l_2}|\leq r-2$, $(l_1,l_2)\neq (1,1)$.  Consequently, we can repeat the argument in \eqref{eqCase2}, obtaining the same estimate.

In summary, the results in Cases 1, 2, and 3 give us
\begin{equation}\label{B weight}
  \Vert \langle x\rangle^{n}\mathscrsfs{B}\Vert_{L^2}\lesssim \sum^{|\beta|}_{p=1}(\eta \,\,\Vert u\Vert_{\X})^{2p} \,\,\Vert u\Vert_{\X_N}^{\alpha+1}.
\end{equation}

$\bullet$ \textbf{ Unweighted $L^2$-norm}. Given $r+1\leq |\beta|\leq M$, we proceed with the estimates of $\mathscrsfs{A}$ and $\mathscrsfs{B}$ in $L^2(\mathbb{R}^N)$. H\"{o}lder's inequality gives us
\begin{equation}\label{A L2}
    \Vert \mathscrsfs{A}\Vert_{L^2}=\Vert |u|^{\alpha}\partial^{\beta}u\Vert_{L^2}\leq \Vert u\Vert_{L^{\infty}}^{\alpha}\,\,\Vert \partial^{\beta} u\Vert_{L^2}\lesssim \Vert u\Vert_{\X_N}^{\alpha+1}.
\end{equation}
On the other hand, following \eqref{B L2 weight}, we obtain
\begin{align*}
\begin{split}
\Vert \mathscrsfs{B}\Vert_{L^2} \lesssim\sum^{|\beta|}_{|\gamma|=1}\sum^{|\gamma|}_{p=1}\sum_{\substack{\gamma_1+\dots+\gamma_p=\gamma,\\|\gamma_j|\geq1}}\sum_{\substack{\gamma_{1,j}+\gamma_{2,j}=\gamma_j,\\1\leq j\leq p}} \eta ^{2p}\,\,\Vert u\Vert^{\alpha}_{\X_N}\,\,\Vert\langle x\rangle^{-n\alpha}\langle x\rangle^{2pn}\,\,\partial^{\beta-\gamma} u \,\,\partial^{\gamma_{1,1}} u\dots\partial^{\gamma_{2,p}} \bar{u}\Vert_{L^2},
     \end{split}
 \end{align*}
and thus, we need to control $\Vert\langle x\rangle^{-n\alpha}\langle x\rangle^{2pn}\,\,\partial^{\beta-\gamma} u \,\,\partial^{\gamma_{1,1}} u\dots\partial^{\gamma_{2,p}} \bar{u}\Vert_{L^2}$. Once again, for simplicity, assume that $\gamma_{1,1}$ is a derivative of the highest order among $\partial^{\beta -\gamma} u$, $\partial^{\gamma_{1,1}} u,\dots,\partial^{\gamma_{2,p}} \bar{u}$. {  We first observe that if all the derivatives are of order less of equal $r-1$ in one dimension and of order less or equal $r-2$ in two and three dimensions, then the proof proceeds exactly as in previous cases. In particular, for dimensions $N=2,3$, we require that $n>\frac{N}{2\alpha}$.} Thus, it is only possible for $\gamma_{1,1}$ to have order $>r-1$ in dimension $N=1$, otherwise, we would have $2r-2<M$, that is, $r<3n+2$, which contradicts our choice of  $r$. Similarly, in dimensions $N=2,3$, only $\gamma_{1,1}$ can have order $>r-2$. This implies that if  $N=1$, all other derivatives among $\partial^{\beta-\gamma} u$, $\partial^{\gamma_{1,2}} u,\dots,\partial^{\gamma_{2,p}} \bar{u}$ must be of order $\leq r-1$, and when $N=2,3$, of order $\leq r-2$. As a consequence, by keeping the higher-order derivative within the $L^2$-norm and applying estimates \eqref{Sobembeddingineq1} and \eqref{Sobembeddingineq2}, we get
\begin{equation*}
\begin{aligned}
    \begin{aligned}
     \Vert\langle x\rangle^{-n\alpha} &\langle x\rangle^{2pn}\,\,\partial^{\beta-\gamma} u \,\,\partial^{\gamma_{1,1}} u\dots\partial^{\gamma_{2,p}} \bar{u}\Vert_{L^2}\\
     &\lesssim \Vert  \partial^{\gamma_{1,1}} u\Vert_{L^2}\,\, 
    \Vert \langle x\rangle^{n}\,\,\partial^{\beta-\gamma} u \Vert_{L^{\infty}}\,\,\Vert\langle x\rangle^{n}\,\,\partial^{\gamma_{1,2}} u\Vert_{L^{\infty}}\cdot \cdot \cdot  \Vert \langle x\rangle^{n}\,\, \partial^{\gamma_{2,p}} \bar{u}\Vert_{L^{\infty}}\\
    &\lesssim \Vert u \Vert_{\X_N}^{2p+1}.
    \end{aligned}  
\end{aligned}    
\end{equation*}
We conclude
\begin{equation*}
  \Vert \mathscrsfs{B}\Vert_{L^2}\lesssim \sum^{|\beta|}_{p=1}(\eta \,\,\Vert u\Vert_{\X})^{2p} \,\,\Vert u\Vert_{\X_N}^{\alpha+1},
\end{equation*}
which completes the study of the $H^M$-norm.

Gathering all the previous estimates, we arrive at
\begin{equation*}
    \Vert |u|^{\alpha}u\Vert_{\X}\lesssim\Vert u \Vert_{\X_N}^{\alpha+1}+\sum^{M}_{p=1}(\eta\Vert u\Vert_{\X_N})^{2p}\Vert u \Vert_{\X_N}^{\alpha+1}\lesssim (1+\eta \Vert u\Vert_{\X_N})^{2M}\Vert u \Vert_{\X_N}^{\alpha+1},
\end{equation*}
which shows \eqref{nonlinear}. 
\smallskip

{\bf Difference estimates.}

Next, we deduce \eqref{nonlinear diffrerence}. Let $u,v\in \X_N$ be such that both of them satisfy \eqref{inf cond sol} for the same $\eta >0$. Given $\theta\in \R$, using a mean value theorem and \eqref{inf cond sol}, we can bound the difference of powers of absolute values as follows
\begin{equation*}
||u|^{\theta}-|v|^{\theta}|\leq 
     \left\{
 \begin{aligned}
 &|\theta|\eta^{-(\theta-1)}\langle x\rangle^{-n(\theta-1)}|u-v|, \hspace{1.75in} \mbox{if} \quad\theta-1<0,\\
 &|\theta|(\max\{\Vert \langle x\rangle^{n}u\Vert_{L^{\infty}}^{\theta-1},\Vert \langle x\rangle^{n}v\Vert_{L^{\infty}}^{\theta-1}\})\langle x\rangle^{-n(\theta-1)}|u-v|,\quad\mbox{if}\quad\theta-1\geq 0. 
 \end{aligned}
 \right.
\end{equation*}
Re-writing it in a more uniform way, we have
\begin{equation}\label{MVT}
||u|^{\theta}-|v|^{\theta}| 
\leq |\theta|(\eta+\Vert u\Vert_{\X_N}+\Vert v\Vert_{\X_N})^{|\theta-1|}\langle x\rangle^{-n(\theta-1)}|u-v|.
\end{equation}

For any multi-index $\beta$, we write
\begin{equation}\label{difference}
    \partial^{\beta}(|u|^{\alpha}u-|v|^{\alpha}v)=:\mathscrsfs{A}_{1}+\mathscrsfs{B}_{1}+\mathscrsfs{B}_{2}+\mathscrsfs{B}_{3},
\end{equation}
where 
\begin{equation*}
 \begin{aligned}
     \mathscrsfs{A}_{1}&:=(|u|^{\alpha}-|v|^{\alpha})\partial^{\beta}u+|v|^{\alpha}\partial^{\beta}(u-v),\\
\mathscrsfs{B}_{1}&:=\sum^{|\beta|}_{|\gamma|=1}\sum^{|\gamma|}_{p=1}\sum_{\substack{\gamma_1+\dots+\gamma_p=\gamma,\\|\gamma_j|\geq1}}\sum_{\substack{\gamma_{1,j}+\gamma_{2,j}=\gamma_j,\\1\leq j\leq p}}\begin{pmatrix}
\beta\\
\gamma
\end{pmatrix}c_{\beta,\gamma_{1,1},\dots,\gamma_{2,p}}\\
&\hspace{2in}\times (|u|^{\alpha-2p}-|v|^{\alpha-2p})\,\,\partial^{\beta-\gamma} u \,\,\partial^{\gamma_{1,1}} u\partial^{\gamma_{2,1}}\overline{u}\dots \partial^{\gamma_{1,p}}u\,\partial^{\gamma_{2,p}} \bar{u}\\
\mathscrsfs{B}_{2}&:=\sum^{|\beta|}_{|\gamma|=1}\sum^{|\gamma|}_{p=1}\sum_{\substack{\gamma_1+\dots+\gamma_p=\gamma,\\|\gamma_j|\geq1}}\sum_{\substack{\gamma_{1,j}+\gamma_{2,j}=\gamma_j,\\1\leq j\leq p}}\begin{pmatrix}
\beta\\
\gamma
\end{pmatrix}c_{\beta,\gamma_{1,1},\dots,\gamma_{2,p}}\\
&\hspace{2in}\times |v|^{\alpha-2p}\,\,\partial^{\beta-\gamma}( u -v)\,\,\partial^{\gamma_{1,1}} u \partial^{\gamma_{2,1}} \overline{u}\dots \partial^{\gamma_{1,p}} u\, \partial^{\gamma_{2,p}} \bar{u}\\
\mathscrsfs{B}_{3}&:=\sum^{|\beta|}_{|\gamma|=1}\sum^{|\gamma|}_{p=1}\sum_{\substack{\gamma_1+\dots+\gamma_p=\gamma,\\|\gamma_j|\geq1}}\sum_{\substack{\gamma_{1,j}+\gamma_{2,j}=\gamma_j,\\1\leq j\leq p}}\sum_{l=1}^{p}\begin{pmatrix}
\beta\\
\gamma
\end{pmatrix}c_{\beta,\gamma_{1,1},\dots,\gamma_{2,p}}\\
&\hspace{2in}\times|v|^{\alpha-2p}\Big(\,\,\partial^{\beta-\gamma} v \,\,\partial^{\gamma_{1,1}} v\partial^{\gamma_{2,1}} \overline{v}\dots\partial^{\gamma_{1,l}}(u-v)\partial^{\gamma_{2,l}}\overline{u}\dots \partial^{\gamma_{1,p}} u\partial^{\gamma_{2,p}} \bar{u}\\
&\hspace{2.5in}+\partial^{\beta-\gamma} v \,\,\partial^{\gamma_{1,1}} v\partial^{\gamma_{2,1}}\overline{v}\dots \partial^{\gamma_{1,l}}v\partial^{\gamma_{2,l}}(\overline{u-v})\dots \partial^{\gamma_{1,p}}u\partial^{\gamma_{2,p}} \bar{u}\Big).
 \end{aligned}   
\end{equation*}
We now proceed with the study of $|u|^{\alpha}u-|v|^{\alpha}v$ in each of the norms in the definition of $\X_N$. 
\smallskip

$\bullet$ \underline{\bf Weighted \bf $L^{\infty}$-norm}. Consider $\beta=0$ if $N=1$, and let $0\leq|\beta|\leq 1$ for $N=2,3$. Applying \eqref{MVT} and definition of the $\X_N$ space, we get
\begin{equation*}\label{A1 inf}
    \begin{aligned}
        \Vert \langle x\rangle^n\mathscrsfs{A}_{1}\Vert_{L^{\infty}}
        &\lesssim \Vert \langle x\rangle^n(|u|^{\alpha}-|v|^{\alpha})\partial^{\beta}u\Vert_{L^{\infty}}
        +\Vert \langle x\rangle^n|v|^{\alpha}\partial^{\beta}(u-v)\Vert_{L^{\infty}}\\
        &\lesssim \big(\eta+\Vert u \Vert_{\X_N}+\Vert v\Vert_{\X_N}\big)^{|\alpha-1|}  \Vert\langle x\rangle^{-n(\alpha-1)}|u-v|\Vert_{L^{\infty}}\Vert \langle x\rangle^n\partial^{\beta}u\Vert_{L^{\infty}}
        +\Vert v\Vert_{\X_N}^{\alpha}\Vert u-v\Vert_{\X_N}\\
       & \lesssim \big(\eta+\Vert u \Vert_{\X_N}+\Vert v\Vert_{\X_N}\big)^{|\alpha-1|}\Vert u-v\Vert_{\X_N}\Vert u\Vert_{\X_N}
     +\Vert v\Vert_{\X_N}^{\alpha}\Vert u-v\Vert_{\X_N}.
         \end{aligned}
\end{equation*}
By \eqref{MVT}, it follows
\begin{equation*}
    \begin{aligned}
      \Vert \langle x\rangle^n\mathscrsfs{B}_{1}\Vert_{L^{\infty}}
        &\lesssim \sum^{|\beta|}_{|\gamma|=1}\sum^{|\gamma|}_{p=1}\sum_{\substack{\gamma_1+\dots+\gamma_p=\gamma,\\|\gamma_j|\geq1}}\sum_{\substack{\gamma_{1,j}+\gamma_{2,j}=\gamma_j,\\1\leq j\leq p}}
\Vert \langle x\rangle^n  (|u|^{\alpha-2p}-|v|^{\alpha-2p})\,\,\partial^{\beta-\gamma} u \,\,\partial^{\gamma_{1,1}} u\dots\partial^{\gamma_{2,p}} \bar{u}\Vert_{L^{\infty}}\\
&\lesssim\sum^{|\beta|}_{|\gamma|=1}\sum^{|\gamma|}_{p=1}\sum_{\substack{\gamma_1+\dots+\gamma_p=\gamma,\\|\gamma_j|\geq1}}\sum_{\substack{\gamma_{1,j}+\gamma_{2,j}=\gamma_j,\\1\leq j\leq p}}\big(\eta+\Vert u\Vert_{\X_N}+\Vert v\Vert_{\X_N}\big)^{|\alpha-2p-1|}\\
&\qquad\qquad\qquad\qquad\qquad\qquad\qquad\times\Vert \langle x\rangle^{-\alpha n+2pn+2n}\,\, |u-v|\,\,\partial^{\beta-\gamma} u \,\,\partial^{\gamma_{1,1}} u\dots\partial^{\gamma_{2,p}} \bar{u}\Vert_{L^{\infty}}\\
&\lesssim \sum^{|\beta|}_{p=1}\big(\eta+\Vert u\Vert_{\X_N}+\Vert v\Vert_{\X_N}\big)^{|\alpha-2p-1|}\Vert u\Vert_{\X_N}^{2p+1}\Vert u-v\Vert_{\X_N}.
    \end{aligned}
\end{equation*}
Similarly, using \eqref{neg power} we get
\begin{equation*}
\begin{aligned}
\Vert \langle x\rangle^{n} \mathscrsfs{B}_{2}\Vert_{L^{\infty}}
&\lesssim \sum^{|\beta|}_{p=1}\eta^{2p}\Vert v\Vert_{\X_N}^{\alpha}\Vert u-v\Vert_{\X_N}\Vert u\Vert_{\X_N}^{2p}\\
&\lesssim \sum^{|\beta|}_{p=1}\eta^{2p}(\Vert u\Vert_{\X_N}+\Vert v\Vert_{\X_N})^{2p}\Vert v\Vert_{\X_N}^{\alpha}\Vert u-v\Vert_{\X_N},
    \end{aligned}
\end{equation*}
and we deduce
\begin{equation*}
  \Vert \langle x\rangle^{n} \mathscrsfs{B}_{3}\Vert_{L^{\infty}} \lesssim  \sum^{|\beta|}_{p=1} \eta^{2p}(\Vert u\Vert_{\X_N}+\Vert v\Vert_{\X_N})^{2p}\Vert v\Vert_{\X_N}^{\alpha}\Vert u-v\Vert_{\X_N}.
\end{equation*}

$\bullet$ \underline{\bf Weighted $L^{2}$-norm}. Here, we take $1\leq |\beta|\leq r$ if $N=1$, and $2\leq |\beta|\leq r$ for $N=2,3$.  By \eqref{MVT}, H\"{o}lder's inequality and definition of the $\X_N$ space, we get
\begin{equation*}
    \begin{aligned}
        \Vert \langle x\rangle^n\mathscrsfs{A}_{1}\Vert_{L^{2}}
        &\lesssim \Vert \langle x\rangle^n(|u|^{\alpha}-|v|^{\alpha})\partial^{\beta}u\Vert_{L^{2}}
        +\Vert \langle x\rangle^n|v|^{\alpha}\partial^{\beta}(u-v)\Vert_{L^{2}}\\
        &\lesssim \big(\eta+\Vert u \Vert_{\X_N}+\Vert v\Vert_{\X_N}\big)^{|\alpha-1|}\Vert  \langle x\rangle^{-n(\alpha-1)}|u-v|\Vert_{L^{\infty}}\Vert \langle x\rangle^n\partial^{\beta}u\Vert_{L^{2}}
        +\Vert v\Vert_{\X_N}^{\alpha}\Vert u-v\Vert_{\X_N}\\
       & \lesssim \big(\eta+\Vert u \Vert_{\X_N}+\Vert v\Vert_{\X_N}\big)^{|\alpha-1|}\Vert u-v\Vert_{\X_N}\Vert u\Vert_{\X_N}
     +\Vert v\Vert_{\X_N}^{\alpha}\Vert u-v\Vert_{\X_N}.
         \end{aligned}
\end{equation*}
Now, using \eqref{MVT}, and the same arguments that lead to \eqref{B weight} (i.e., applying the same ideas as in the Cases 1, 2, and 3 above), we deduce
\begin{equation*}
\begin{aligned}
\Vert \langle x\rangle^n\mathscrsfs{B}_{1}\Vert_{L^{2}}
&\lesssim \sum^{|\beta|}_{|\gamma|=1}\sum^{|\gamma|}_{p=1}\sum_{\substack{\gamma_1+\dots+\gamma_p=\gamma,\\|\gamma_j|\geq1}}\sum_{\substack{\gamma_{1,j}+\gamma_{2,j}=\gamma_j,\\1\leq j\leq p}}
\Vert \langle x\rangle^n  (|u|^{\alpha-2p}-|v|^{\alpha-2p})\,\,\partial^{\beta-\gamma} u \,\,\partial^{\gamma_{1,1}} u\dots\partial^{\gamma_{2,p}} \bar{u}\Vert_{L^{2}}\\
&\lesssim\sum^{|\beta|}_{|\gamma|=1}\sum^{|\gamma|}_{p=1}\sum_{\substack{\gamma_1+\dots+\gamma_p=\gamma,\\|\gamma_j|\geq1}}\sum_{\substack{\gamma_{1,j}+\gamma_{2,j}=\gamma_j,\\1\leq j\leq p}}
\big(\eta+\Vert u\Vert_{\X_N}+\Vert v\Vert_{\X_N}\big)^{|\alpha-2p|}\\
&\hspace{2in}\times\Vert u-v\Vert_{\X_N}\Vert \langle x\rangle^{-\alpha n+2pn+n}\,\,\partial^{\beta-\gamma} u \,\,\partial^{\gamma_{1,1}} u\dots\partial^{\gamma_{2,p}} \bar{u}\Vert_{L^{2}}\\
&\lesssim \sum^{|\beta|}_{p=1}\big(\eta+\Vert u\Vert_{\X_N}+\Vert v\Vert_{\X_N}\big)^{|\alpha-2p-1|}\Vert u\Vert_{\X_N}^{2p+1}\Vert u-v\Vert_{\X_N}.
    \end{aligned}
\end{equation*}
Analogously, applying the ideas in the deduction of \eqref{B weight}, we find
\begin{equation*}
\begin{aligned}
\Vert \langle x\rangle^{n} \mathscrsfs{B}_{2}\Vert_{L^2}+ \Vert \langle x\rangle^{n} \mathscrsfs{B}_{3}\Vert_{L^{2}} \lesssim \sum^{|\beta|}_{p=1}\eta^{2p}(\Vert u\Vert_{\X_N}+\Vert v\Vert_{\X_N})^{2p}\Vert v\Vert_{\X_N}^{\alpha}\Vert u-v\Vert_{\X_N}.
    \end{aligned}
\end{equation*}


$\bullet$ \underline{\bf  Unweighted $L^2$-norm}. Taking $r \leq |\beta| \leq M$, we need to control the $L^2$-norm of $\mathscrsfs{A}_1$ and $\mathscrsfs{B}_j$, $j = 1, 2, 3$. These estimates follow from \eqref{MVT}, and the definitions of $r$ and $M$, which imply that only one derivative can have order greater than $r-1$ in dimension $N = 1$, and order less than $r-2$ in dimensions $N = 2, 3$. In summary, we have

 \begin{equation}\label{dif L2}
 \begin{aligned}
     \Vert \partial^{\beta}(|u|^{\alpha}u-|v|^{\alpha}v)\Vert_{L^{2}}
  & \lesssim \big(\eta+\Vert u \Vert_{\X_N}+\Vert v\Vert_{\X_N}\big)^{|\alpha-1|}\Vert u-v\Vert_{\X_N}\Vert u\Vert_{\X_N}
     +\Vert v\Vert_{\X_N}^{\alpha}\Vert u-v\Vert_{\X_N}\\
     &+\sum^{|\beta|}_{p=1}\Big(\big(\eta+\Vert u\Vert_{\X_N}+\Vert v\Vert_{\X_N}\big)^{|\alpha-2p-1|}\Vert u\Vert_{\X_N}^{2p+1}\Vert u-v\Vert_{\X_N}\\
     &\qquad\qquad\qquad+\eta^{2p}(\Vert u\Vert_{\X_N}+\Vert v\Vert_{\X_N})^{2p}\Vert v\Vert_{\X_N}^{\alpha}\Vert u-v\Vert_{\X_N}\Big).
\end{aligned}
\end{equation}
Using that $\eta\leq (\eta+\|u\|_{\X_N}+\|v\|_{\X_N})$ and $(\Vert u\Vert_{\X_N}+\Vert v\Vert_{\X_N})\leq (\eta+\Vert u\Vert_{\X_N}+\Vert v\Vert_{\X_N})$, we gather the previous estimates for $\mathscrsfs{A}_1$, $\mathscrsfs{B}_1$, $\mathscrsfs{B}_2$, and $\mathscrsfs{B}_3$ to complete the deduction of \eqref{nonlinear diffrerence}.  
\end{proof}


\subsection{Proof of Theorem \ref{theorem1}}
\label{S:proof}

In this part, we combine our linear and nonlinear estimates to deduce well-posedness results in the $\X_N$ space.

\begin{proof}[Proof of Theorem \ref{theorem1}] To deduce the existence of solutions, we apply the contraction mapping principle to the integral formulation of \eqref{biNLS}. Since the initial condition $u_0\in \mathcal{X}_N$ satisfies \eqref{inf cond u0}, we take $\eta>0$ such that
\begin{equation*}
  \inf_{x\in \mathbb{R}^N}\langle x\rangle^n |u_0(x)|\geq \frac{2}{\eta}.  
\end{equation*}
Let $K>0$ and $T>0$ be chosen later, and define space $\mathfrak{E}_{K,T}$ by
\begin{equation*}
    \mathfrak{E}_{K,T}=\left\{f\in C([-T,T];\X_N);\sup_{t\in[-T,T]}\Vert f(t)\Vert_{\X_N}\leq K, \quad \inf_{(x,t)\in \R^N \times [-T,T]}|\langle x\rangle^{n}f(x,t)|\geq\frac{1}{\eta}\right\}.
\end{equation*}
It is not hard to check that $ \mathfrak{E}_{K,T}$ equipped with distance function $d(u,v)=\sup_{t\in[-T,T]}\Vert u-v\Vert_{\X_N}$ is a complete metric space. Inspired by the integral formulation of \eqref{biNLS}, we consider the operator
\begin{equation*}
    \Phi(u(t))=e^{-it(\Delta^{2}+2a\Delta)}u_0+i\lambda\int_0^t e^{-i(t-\tau)(\Delta^2+2a\Delta)}|u|^{\alpha}u \,\,d\tau,
\end{equation*}
with $u\in \X_N$. Applying Lemma \ref{pure bi-laplacian} when $a=0$ or Lemma \ref{linear est l-laplacian} when $a\neq 0$, together with \eqref{nonlinear}, we bound for $|t| <T$
\begin{align}
\begin{split}\label{5.3}
\Vert \Phi (u(t))\Vert_{\X_N} &\leq c \langle t\rangle^{2n+1}\Vert u_0\Vert_{\X_N}+c|\lambda|\Big|\int_0^{t} \langle t-\tau\rangle^{2n+1}\Vert |u|^{\alpha}u (\tau)\Vert_{\X_N}\,\,d\tau\Big|\\
&\leq c \langle t\rangle^{2n+1}\Vert u_0\Vert_{\X_N}+c |\lambda| \sup_{t\in[-T,T]}\Big(|t|\langle t\rangle^{2n+1}(1+\eta\Vert u(t)\Vert_{\X_N})^{2M}\Vert u(t)\Vert_{\X_N}^{\alpha+1}\Big)\\
&\leq c \langle T\rangle^{2n+1}\Vert u_0\Vert_{\X_N}+c|\lambda|T\langle T\rangle^{2n+1}(1+\eta K)^{2M}\,K^{\alpha+1},
\end{split}
\end{align}
for some constant $c>0$. Setting $K=c\,2^{2n+2}\Vert u_0\Vert_{\X_N}$, and choosing  $0<T<1$ small such that
\begin{equation*}
   c|\lambda|T\langle T\rangle^{2n+1}(1+\eta K)^{2M}\,K^{\alpha+1}\leq c2^{2n+1}\Vert u_0\Vert_{\X_N}, 
\end{equation*}
we have $\sup\limits_{t\in[-T,T]}\Vert\Phi(u(t))\Vert_{\X_N}\leq K$, implying that  $\Phi(u)\in C([-T,T];\X_N)$. 

On the other hand, by Lemmas \ref{pure bi-laplacian}, \ref{linear est l-laplacian} and \ref{nonlinear est}, we obtain
\begin{equation*}
    \begin{aligned}
     \Vert \Phi (u(t))-\Phi(v(t))\Vert_{\X_N}\leq &  c|\lambda|\Big|\int_0^{t}\langle t-\tau \rangle^{2n+1}\Vert (|u|^{\alpha}u-|v|^{\alpha}v)\Vert_{\X_N}\,\,d\tau\Big|\\ 
     \leq & c|\lambda|T \langle T\rangle^{2n+1}\mathcal{G}(\eta,K,K)\sup_{t\in[-T,T]}\Vert u(t)-v(t)\Vert_{\X_N},    
    \end{aligned}
\end{equation*}
where $\mathcal{G}(\eta,\cdot,\cdot)$ is given by \eqref{nonlinearGfunc}. Thus, taking $T>0$ small such that 
\begin{equation*}
   c|\lambda|T\langle T\rangle^{2n+1}\mathcal{G}(\eta,K,K)\leq\frac{1}{2},
\end{equation*}
it follows 
\begin{equation}\label{5.7}
    \sup_{t\in[-T,T]}\Vert \Phi (u(t))-\Phi(v(t))\Vert_{\X_N}\leq\frac{1}{2}\sup_{t\in[-T,T]}\Vert u(t)-v(t)\Vert_{\X_N}.
\end{equation}
Next, we check the non-vanishing condition for $\Phi(u)$. From \eqref{diffsol1}, \eqref{diffsol2}, and \eqref{5.3}, there exists a constant $c>0$ such that
\begin{align*}
   \begin{split}
       |\langle x \rangle^{n}\Phi(u(t))(x)|\geq &|\langle x \rangle^{n}u_0|-\Vert \langle x \rangle^{n}\big(e^{-it(\Delta^{2}+2a\Delta)}u_0-u_0\big)\Vert_{L^{\infty}}
       -\Big|\int_0^{t}\Vert e^{-i(t-\tau)(\Delta^{2}+2a\Delta)} |u|^{\alpha}u\Vert_{\X_N}\,\,d\tau \Big|\\
       \geq & \frac{2}{\eta}-c|T|\langle T\rangle^{2n+1}K-c|\lambda|T\langle T \rangle^{2n+1}(1+\eta K)^{2M}K^{\alpha+1}.
   \end{split}
\end{align*}
Thus, taking $T>0$ small such that 
\begin{equation*}
 c|T|\langle T\rangle^{2n+1}K+c|\lambda|T\langle T \rangle^{2n+1}(1+\eta K)^{2M}K^{\alpha+1}\leq \frac{1}{\eta},
\end{equation*}
it follows that $|\langle x \rangle^{n}\Phi(u(t))(x)|\geq \frac{1}{\eta}$ for all $x\in \mathbb{R}^N$, and $t\in [-T,T]$.

Consequently,  the previous estimates show that there exists a time $0<T<1$ such that $\Phi:\mathfrak{E}_{K,T}\longrightarrow\mathfrak{E}_{K,T}$ define a contraction map on the complete metric space $\mathfrak{E}_{K,T}$. Accordingly, Banach fixed-point theorem assures the existence of a solution to the integral formulation of \eqref{biNLS}. This completes the existence part.
\smallskip

We next deduce the uniqueness of solutions. Suppose that there exist two solutions $u, v \in C([-T,T];\\ \X_N)$ of \eqref{biNLS} with the same initial condition $u_0\in \X_N$, where $u_0$ satisfies the infimum condition \eqref{inf cond u0}, and $u$ and $v$ satisfy \eqref{Xinf}. Fix $K_1>0$ such that 
$$
\sup_{t\in[-T,T]}(\Vert u(t)\Vert_{\X_N}+\Vert v(t)\Vert_{\X_N})\leq K_1,
$$
and let $\eta>0$ be such that
\begin{equation*}
   \inf_{(x,t)\in \R^N \times [-T,T]}|\langle x\rangle^{n}u(x,t)|\geq \frac{1}{\eta}\qquad \text{and} \qquad \inf_{(x,t)\in \R^N \times [-T,T]}|\langle x\rangle^{n}v(x,t)|\geq \frac{1}{\eta}.
\end{equation*}
Note that our regularity assumptions on space $\X_N$ assure that $u,v$ satisfy \eqref{biNLS} pointwise. Thus, it follows that $w=u-v$ solves the equation
\begin{equation}\label{w}
   i\partial_{t} {w} -\Delta^{2} {w}-2a\Delta w+\lambda(|u|^{\alpha}-|v|^{\alpha})u+\lambda|v|^{\alpha}(u-v)=0.
\end{equation}
Multiplying \eqref{w} by $\bar{w}$, integrating the resulting equation on space, and taking the imaginary part, gives us
\begin{equation}\label{5.10}
    \begin{aligned}
    \frac{1}{2}\frac{d}{dt}\int |w|^2\, dx&-\underbrace{\Im\int \Delta^{2} {w}\bar{w}\, dx}_{\mathscrsfs{I}_1}-\underbrace{\Im\int 2a \Delta {w}\bar{w}\, dx}_{\mathscrsfs{I}_2}\\
     &+\underbrace{\Im\int\lambda(|u|^{\alpha}-|v|^{\alpha})u\bar{w}\,dx}_{\mathscrsfs{I}_3}+\underbrace{\Im\int\lambda|v|^{\alpha}|w|^2\,dx}_{\mathscrsfs{I}_4}=0.
    \end{aligned}
\end{equation}
Integration by parts shows that $\mathscrsfs{I}_1$ and $\mathscrsfs{I}_2$ have null contribution to the identity above. Now, using \eqref{MVT}, there exists a constant $c>0$ such that 
\begin{equation*}
\begin{aligned}
    |\mathscrsfs{I}_3|&\leq c|\lambda||\alpha| \big(\eta+2K_1\big)^{|\alpha-1|} K_1\int |w|^2\,dx,
\end{aligned}
\end{equation*}
and an estimate for $\mathscrsfs{I}_4$ follows immediately. Summarizing, we obtain that for some positive constant $c>0$, 
\begin{equation*}
    \frac{1}{2} \frac{d}{dt}\int |w|^2\, dx\leq c|\lambda|\big(|\alpha| \big(\eta+2K_1\big)^{|\alpha-1|} K_1+K_1^{\alpha}\big)\int|w|^2\, dx.
\end{equation*}
Hence, Gronwall's inequality, and the fact that $w(0)=u(0)-v(0)=0$, gives us the uniqueness result. Finally, the continuity of the mapping data-to-solution of \eqref{biNLS} follows from the fact that the existence time of solution is a continuous function of the initial condition, and applying the same argument as in the deduction of \eqref{5.7} above. Thus, the proof of Theorem \ref{theorem1} is now complete.
\end{proof}

\section{Local well-posedness in \texorpdfstring{$H^s$}{Hs} spaces}

In this part, we focus on the deduction of unconditional well-posedness for the Cauchy problem associated to \eqref{biNLS}. We first present some preliminaries on fractional derivative estimates, and then describe how to obtain linear and nonlinear estimates that yield the proof of Theorem \ref{LWPresults}.

\subsection{Preliminaries on fractional derivatives}\label{S:Hs}

We will use the following fractional Leibniz rule estimate, the proof of which can be found in  \cite{GrafakosOh2014} and references therein.
\begin{lemma} \label{Grafakos-Oh}
Let $s>0$, $1<p<\infty$ and $p_1,p_2,q_1,q_2\in (1,\infty]$ be such that $\frac{1}{p}=\frac{1}{p_1} + \frac{1}{q_1} =\frac{1}{p_2} + \frac{1}{q_2} $. Then 
\begin{equation*}
   \|D^s(fg)\|_{L^{p}(\mathbb{R}^N)} \lesssim \|D^s f\|_{L^{p_1}(\mathbb{R}^N)} \|g\|_{L^{q_1}(\mathbb{R}^N)} + \|f\|_{L^{p_2}(\mathbb{R}^N)}\|D^s g\|_{L^{q_2}(\mathbb{R}^N)}.
\end{equation*}
\end{lemma}

We also need some fractional derivative results. We denote by $L^{p}_{b}(\mathbb{R}^d)$ the fractional Sobolev space defined as $L^{p}_{b}(\mathbb{R}^N):=(1-\Delta)^{-\frac{b}{2}}L^{p}(\R^N)$. The proof of the following result can be found in \cite{Stein1961}.
\begin{theorem}\label{stein}
Let $b\in (0,1)$ and $\frac{2N}{N+2b}<p<\infty.$ Then $f\in L^{p}_{b}(\R^{N})$ if and only if
\begin{itemize}
		\item [a)] $f\in L^{p}(\R^{N}),$ 
		\item [b)]
		$\mathcal{D}^{b}f(x):={\displaystyle \left (
			\int_{\R^{N}}\frac{|f(x)-f(y)|^{2}}{|x-y|^{N+2b}}dy\right)^{\frac{1}{2}}} \in
		L^{p}(\R^{N}),$ with
		\begin{equation*}
		\|J^{b}f\|_{L^p}\sim \big(\|f\|_{L^p}+\|D^{b}f\|_{L^p}\big)\sim \|f\|_{L^p}+\|\mathcal{D}^{b}f\|_{L^p}.
\end{equation*}
In particular, when $p=2$, $\|D^b f\|_{L^2}\sim \|\mathcal{D}^bf\|_{L^2}$.
\end{itemize}
\end{theorem}

As a consequence of the previous result, we can apply Fubini's theorem to obtain 
\begin{equation}\label{Leib}
\|\mathcal{D}^{b}(fg)\|_{L^2(\R^N)} \leq \|f\mathcal{D}^{b}g\|_{L^2(\R^N)} + \|g\mathcal{D}^{b}f\|_{L^2(\R^N)}.
\end{equation}
The following estimates are valid for the appropriate functions $h$ and $f$,
\begin{equation}\label{Lei}
\mathcal{D}^b h(x)\lesssim \|h\|_{L^{\infty}(\R^N)}+\|\nabla h\|_{L^\infty(\R^N)}
\end{equation}
and
\begin{equation}\label{Leibh}
\mathcal{D}^b(fh)(x)\lesssim |f(x)|\mathcal{D}^b(h)(x)+\|h\|_{L^{\infty}}\mathcal{D}^b(f)(x).
\end{equation}

We next derive the estimates of the fractional derivative of the nonlinear term $|u|^{\alpha}u$.
\begin{lemma}\label{lemmafracderNON}
Let $\alpha>0$ and $0<\theta<1$. It follows
 \begin{equation}\label{fracderivalpha}
  \mathcal{D}^{\theta}\big(|u|^{\alpha} u\big)(x)\lesssim \|u\|_{L^{\infty}}^{\alpha}\mathcal{D}^{\theta}(u)(x).
 \end{equation}   
Moreover, 
 \begin{equation}\label{fracderivalphaD}
 \begin{aligned}
   \mathcal{D}^{\theta}\big(|u|^{\alpha} u-|v|^{\alpha} v\big)(x) \lesssim &(\|u\|_{L^{\infty}}^{\alpha}+\|v\|_{L^{\infty}}^{\alpha}\big)\mathcal{D}^{\theta}(u-v)(x)\\
&+\left\{ \begin{aligned}
&\ F_{\theta}(x)\|u-v\|_{L^{\infty}}^{\alpha},  \hspace{1.37in} \text{if $0<\alpha <1$}, \\
&\big(\|u\|_{L^{\infty}}^{\alpha-1}+\|v\|_{L^{\infty}}^{\alpha-1}\big) F_{\theta}(x)\|u-v\|_{L^{\infty}}, \hspace{0.25in} \text{if } \alpha \geq 1,   
\end{aligned}\right.       
 \end{aligned}
 \end{equation} 
where
\begin{equation*}
  F_{\theta}(x)=\mathcal{D}^{\theta}(u)(x)+\mathcal{D}^{\theta}(v)(x).
\end{equation*}
\end{lemma}
Before deducing Lemma \ref{lemmafracderNON}, we establish an auxiliary result that will be useful for estimating fractional derivatives of $\partial^{\beta}(|u|^{\alpha}u)$ for $|\beta|\geq 1$. But first, we introduce some notation. Let $\nu>0$ and $p\in \mathbb{Z}^{+}\cup\{0\}$. Noting that $|z|^{\nu} = |z|^{\nu-p}|z|^p$, we denote a generalization of this by
\begin{equation*}
G_{\nu,p}(z,\overline{z}) :=(z\overline{z})^{\frac{\nu}{2}-\frac{p}{2}}\prod_{j=1}^p w_j,    
\end{equation*}
where each $w_j$ is a factor equal to (the first power of ) either $z$ or $\overline{z}$, and we assume that the empty product is $1$, i.e., $\prod_{j=1}^{0}(\cdots)=1$. We also denote by $G_{\nu,p}'(z,\overline{z})=(\partial_z G_{\nu,p}(z,\overline{z}),\partial_{\overline{z}} G_{\nu,p}(z,\overline{z}))$. To simplify notation, we denote by $G_{\nu,p}(z)=G_{\nu,p}(z,\overline{z})$ and $G_{\nu,p}'(z)=G_{\nu,p}'(z,\overline{z})$. Note that in particular, if $u$ is a complex-valued function, then $|G_{\nu,p}(u)|=|u|^{\nu}$.
\begin{lemma}\label{GlemmafracderNON}
Let $\nu>0$ and $0<\theta<1$. It follows that
 \begin{equation}\label{fracderiv}
  \mathcal{D}^{\theta}\big(G_{\nu,p}(u)\big)(x)\lesssim \left\{ \begin{aligned}
 &\|u\|_{L^{\infty}}^{\nu-1}\mathcal{D}^{\theta}(u)(x), \hspace{1.04in} \text{for }\nu\geq 1, \\
 &\left(\int_{\mathbb{R}^N}\frac{|u(x)-u(y)|^{2\nu}}{|x-y|^{2+2\theta}}\, dy\right)^{\frac{1}{2}}, \hspace{0.3in} \text{for } 0<\nu<1.
 \end{aligned}\right.
 \end{equation}   
Moreover, if $\nu\geq 1$,
\begin{equation}\label{fracderivdiff}
\begin{aligned}
\mathcal{D}^{\theta}\Big(G_{\nu,p}(u)-G_{\nu,p}(v)\Big)(x)\lesssim &(\|u\|_{L^{\infty}}^{\nu-1}+\|v\|_{L^{\infty}}^{\nu-1}\big)\mathcal{D}^{\theta}(u-v)(x)\\
&+\left\{ \begin{aligned}
&\big( F_{\theta}(x)\|u-v\|_{L^{\infty}}^{\nu-1},  \hspace{1.3in} \text{if $1\leq \nu <2$}, \\
&\big(\|u\|_{L^{\infty}}^{\nu-2}+\|v\|_{L^{\infty}}^{\nu-2}\big) F_{\theta}(x)\|u-v\|_{L^{\infty}}, \hspace{0.25in} \text{if } \nu \geq 2,   
\end{aligned}\right.    
\end{aligned}
\end{equation}
where
\begin{equation*}
  F_{\theta}(x)=\mathcal{D}^{\theta}(u)(x)+\mathcal{D}^{\theta}(v)(x).
\end{equation*}
\end{lemma}

Given that the proofs of Lemmas \ref{lemmafracderNON} and \ref{GlemmafracderNON} proceed along similar arguments, we provide a more detailed explanation of the second one.

\begin{proof}[Proof of Lemma \ref{GlemmafracderNON}]
We first deduce \eqref{fracderiv}. Taking square power, we divide our estimates as follows
\begin{equation*}
\begin{aligned}
 \Big( \mathcal{D}^{\theta}\big(G_{\nu,p}(u)\big)(x)\Big)^2=&\int_{\mathbb{R}^N}\frac{|G_{\nu,p}(u(x))-G_{\nu,p}(u(y))|^2}{|x-y|^{N+2\theta}}\, dy\\
  =&\int_{|u(y)|\leq |u(x)|}(\dots)\, dy+\int_{|u(x)|\leq |u(y)|}(\dots)\, dy\\
  =:&\mathcal{I}_1+\mathcal{I}_2.
\end{aligned}
\end{equation*}    
We first deal with $\mathcal{I}_1$, i.e., the region $|u(y)|\leq |u(x)|$. We may further assume that $u(y)\neq 0$. Otherwise, 
\begin{equation*}
|G_{\nu,p}(u(x))-G_{\nu,p}(u(y))|=|G_{\nu,p}(u(x))|=|u(x)|^{\nu}=|u(x)-u(y)|^{\nu},    
\end{equation*}
which yields the desired estimates. Thus, assuming that $0<|u(y)|\leq |u(x)|$, we write
\begin{equation*}
\begin{aligned}
G_{\nu,p}(u(x))-G_{\nu,p}(u(y))=&\big(|u(x)|^{\nu-p}-|u(y)|^{\nu-p}\big)\prod_{j=1}^{p}v_j(y)+|u(x)|^{\nu-p}\Big(\prod_{j=1}^{p}v_j(x)-\prod_{j=1}^{p}v_j(y)\Big) \\
 =&\big(|u(x)|^{\nu-p}-|u(y)|^{\nu-p}\big)\prod_{j=1}^{p}v_j(y)\\
 &+|u(x)|^{\nu-p}\sum_{l=1}^p\Big(v_1(y)\dots v_{l-1}(y)(v_l(x)-v_l(y))v_{l+1}(x)\dots v_{p}(x)\Big)\\
 =:&\mathcal{I}_{1,1}+\mathcal{I}_{1,2},
\end{aligned}    
\end{equation*}
where when $l=1$, we understand that the first term of the sum above is $ (v_1(x)-v_1(y))v_2(x)\dots\dots v_{p}(x)$ and when $l=p$, this factor is $v_1(y)\dots v_{p-1}(y)(v_{p}(x)-v_{p}(y))$. To estimate $\mathcal{I}_{1,1}$ and $\mathcal{I}_{1,1}$, we consider two cases: $\nu \geq 1$ and $0<\nu<1$.

\underline{\bf Case $\nu\geq 1$}. Since $|u(y)|\leq |u(x)|$, and $|v_j(y)|=|u(y)|$, the mean value inequality yields
\begin{equation*}
\begin{aligned}
  |\mathcal{I}_{1,1}|\lesssim 
\left\{\begin{aligned}
&|u(x)|^{\nu-1}|u(x)-u(y)|, \quad \text{ when }\quad \nu-p-1\geq 0,\\
&|u(y)|^{\nu-1}|u(x)-u(y)|, \quad \text{ when }\quad \nu-p-1< 0.
\end{aligned}  \right.
\end{aligned}    
\end{equation*}
It follows
\begin{equation*}
 |\mathcal{I}_{1,1}|\lesssim \|u\|_{L^{\infty}}^{\nu-1}|u(x)-u(y)|.   
\end{equation*}
The same inequality above holds for $\mathcal{I}_{1,2}$, given that $|u(y)|^{l-1}\lesssim |u(x)|^{l-1}$ for all $l=1,\dots,p$ implies 
\begin{equation*}
 \begin{aligned}
 |\mathcal{I}_{1,1}|\lesssim & \sum_{l=1}^p|u(x)|^{\nu-p}|u(y)|^{l-1}|u(x)|^{p-l}|u(x)-u(y)|  \\
 \lesssim & |u(x)|^{\nu-1}|u(x)-u(y)|\\
 \lesssim &\|u\|_{L^{\infty}}^{\nu-1}|u(x)-u(y)|.
 \end{aligned}   
\end{equation*}
\underline{\bf Case $0<\nu<1$}. We write
\begin{equation*}
\begin{aligned}
\mathcal{I}_{1,1}=&\big(|u(x)|^{\nu}-|u(y)|^{\nu}\big)|u(x)|^{-p}\prod_{j=1}^{p}v_j(y)+|u(y)|^{\nu}\big(|u(x)|^{-p}-|u(y)|^{-p}\big)\prod_{j=1}^{p}v_j(y)\\
 =&\big(|u(x)|^{\nu}-|u(y)|^{\nu}\big)|u(x)|^{-p}\prod_{j=1}^{p}v_j(y)+|u(y)|^{\nu}\Big(\frac{|u(y)|^p-|u(x)|^p}{|u(x)|^{p}|u(y)|^p}\Big)\prod_{j=1}^{p}v_j(y).  
\end{aligned}
 \end{equation*}
Using the mean value inequality and that $|u(x)-u(y)|\lesssim |u(x)|^{1-\nu}|u(x)-u(y)|^{\nu}$ (provided that $|u(y)|\leq |u(x)|$), we get
\begin{equation*}
    \begin{aligned}
     \Big|\frac{|u(y)|^p-|u(x)|^p}{|u(x)|^{p}|u(y)|^p}\Big| \lesssim & |u(x)|^{-1}|u(y)|^{p}|u(x)-u(y)|\\
     \lesssim &|u(x)|^{-\nu}|u(y)|^{-p}|u(x)-u(y)|^{\nu}.
    \end{aligned}
\end{equation*}
We conclude from the above estimate and H\"older continuity of the function $|x|^{\nu}$, $0<\nu<1$, that
\begin{equation*}
    |\mathcal{I}_{1,1}|\lesssim |u(x)-u(y)|^{\nu}.
\end{equation*}
Once again using that in the present case $|u(x)-u(y)|\lesssim |u(x)|^{1-\nu}|u(x)-u(y)|^{\nu}$, we deduce
\begin{equation*}
    |\mathcal{I}_{1,2}|\lesssim |u(x)-u(y)|^{\nu}.
\end{equation*}
This concludes the case $0<\nu<1$.

Back to the integral $\mathcal{I}_1$, the previous estimates give us the following. 
As we have estimated the difference $|G_{\nu,p}(u(x))-G_{\nu,p}(u(y))|$, we divide it by $|x-y|^{N+2\theta}$ and integrate over the support of $\mathcal{I}_1$. For example, when $\nu\geq 1$, we get 
\begin{equation*}
\begin{aligned}
\mathcal{I}_1=\int_{|u(y)|\leq |u(x)|}\frac{|G_{\nu,p}(u(x))-G_{\nu,p}(u(y))|^2}{|x-y|^{N+2\theta}}\, dy\lesssim & \|u\|_{L^{\infty}}^{2\nu-2} \int_{|u(y)|\leq |u(x)|}\frac{|u(x)-u(y)|^2}{|x-y|^{N+2\theta}}\, dy \\
\lesssim &   \|u\|_{L^{\infty}}^{2\nu-2}\big(\mathcal{D}^{\theta}(u)(x)\big)^2 .    
\end{aligned}
\end{equation*}
Thus, in both cases, we have
\begin{equation*}
 \mathcal{I}_1\lesssim 
\left\{ \begin{aligned}
 &\|u\|_{L^{\infty}}^{2\nu-2}\big(\mathcal{D}^{\theta}(u)(x)\big)^2, \hspace{0.5in} \text{for }\nu\geq 1, \\
 &\int_{\mathbb{R}^N}\frac{|u(x)-u(y)|^{2\nu}}{|x-y|^{2+2\theta}}\, dy, \hspace{0.33in} \text{for } 0<\nu<1.
 \end{aligned}\right.
\end{equation*}
By symmetry, one can change the order of the variables $x$ and $y$ in the previous argument to show that $\mathcal{I}_2$ also satisfies the last inequality above. The proof of \eqref{fracderiv} is now complete.

We now  establish \eqref{fracderivdiff}. We write
\begin{equation*}
\begin{aligned}
G_{\nu,p}(u(x))&-G_{\nu,p}(v(x))-\Big(G_{\nu,p}(u(y))-G_{\nu,p}(v(y))\Big)\\
=&\int_{0}^1 G_{\nu,p}'\big(u(y)+\theta(u(x)-u(y))\big)\bigg(\begin{array}{cc}
     u(x)-u(y)  \\
     \overline{u}(x)-\overline{u}(y)
\end{array}\bigg)\, d\theta\\
&-\int_{0}^1 G_{\nu,p}'\big(v(y)+\theta(v(x)-v(y))\big)\bigg(\begin{array}{cc}
     v(x)-v(y)  \\
     \overline{v}(x)-\overline{v}(y)
\end{array}\bigg)\, d\theta\\
=&\int_{0}^1 G_{\nu,p}'\big(u(y)+\theta(u(x)-u(y))\big)\bigg(\begin{array}{cc}
     u(x)-u(y)-(v(x)-v(y))  \\
     \overline{u(x)}-\overline{u(y)}-(\overline{v}(x)-\overline{v}(y))
\end{array}\bigg)\, d\theta\\
&+\int_0^1 \Big(G_{\nu,p}'\big(u(y)+\theta(u(x)-u(y))\big)-G_{\nu,p}'\big(v(y)+\theta(v(x)-v(y))\big)\Big)\bigg(\begin{array}{cc}
     v(x)-v(y)  \\
     \overline{v}(x)-\overline{v}(y)
\end{array}\bigg)\, d\theta\\
=:&\mathcal{G}_1+\mathcal{G}_2.
\end{aligned}    
\end{equation*}
Since $|G_{\nu,p}'(z)|\lesssim |z|^{\nu-1}$, we deduce
\begin{equation*}
\begin{aligned}
 |\mathcal{G}_1|\lesssim & \big(|u(x)|^{\nu-1}+|u(y)|^{\nu-1}\big)|u(x)-v(x)-\big(u(y)-v(y)\big)|\\
 \lesssim& \big(\|u\|_{L^{\infty}}^{\nu-1}+\|v\|^{\nu-1}_{L^{\infty}}\big)|u(x)-v(x)-\big(u(y)-v(y)\big)|.
\end{aligned}    
\end{equation*}
Now, arguing as in the  estimate for $\mathcal{I}_{1,1}$ and $\mathcal{I}_{1,2}$ above, it follows
\begin{equation*}
|\mathcal{G}_2|\lesssim \left\{\begin{aligned}
&\big(|u(y)-v(y)|^{\nu-1}+|u(x)-v(x)-(u(y)-v(y))|^{\nu-1}\big)|v(x)-v(y)|, \hspace{0.2in} \text{if $0\leq \nu-1<1$}, \\ &  \big(|u(y)|^{\nu -2}+|v(y)|^{\nu-2}+|u(x)-u(y)|^{\nu -2}+|v(x)-v(y)|^{\nu-2} \big) \\
& \hspace{0.2in}\times \Big(|u(y)-v(y)|+|u(x)-v(x)-(u(y)-v(y))|\Big)|v(x)-v(y)|, \hspace{0.23in} \text{if $\nu-1\geq 1$}.
\end{aligned}\right.   
\end{equation*}
Taking $L^{\infty}$-norm, we use the previous inequality to get
\begin{equation*}
|\mathcal{G}_2|\lesssim \left\{\begin{aligned}
&\|u-v\|_{L^{\infty}}^{\nu-1}|v(x)-v(y)|, \hspace{1.36in} \text{if $0\leq \nu-1<1$}, \\ &  \big(\|u\|_{L^{\infty}}^{\nu -2}+\|v\|_{L^{\infty}}^{\nu-2} \big) \|u-v\|_{L^{\infty}}|v(x)-v(y)|, \hspace{0.23in} \text{if $\nu-1\geq 1$}.
\end{aligned}\right.   
\end{equation*}
Similar to the above (an estimate of $\mathcal I_1$), we deduce
\begin{equation*}
\begin{aligned}
\mathcal{D}^{\theta}\Big(G_{\nu,p}(u)-G_{\nu,p}(v)\Big)(x)\lesssim & (\|u\|_{L^{\infty}}^{\nu-1}+\|v\|_{L^{\infty}}^{\nu-1}\big)\mathcal{D}^{\theta}(u-v)(x)\\
&+\left\{ \begin{aligned}
&\big(\|u-v\|_{L^{\infty}}^{\nu-1} \mathcal{D}^{\theta}(v)(x),  \hspace{1.3in} \text{if $0\leq \nu-1<1$}, \\
&\big(\|u\|_{L^{\infty}}^{\nu-2}+\|v\|_{L^{\infty}}^{\nu-2}\big)\|u-v\|_{L^{\infty}} \mathcal{D}^{\theta}(v)(x), \hspace{0.25in} \text{if } \nu-1\geq 1,   
\end{aligned}\right. 
\end{aligned} 
\end{equation*}
which finished the proof. 
\end{proof}

\begin{proof}[Proof of Lemma \ref{fracderivalpha}]
The inequality \eqref{fracderivalpha} follows the same steps as in the deduction of \eqref{fracderiv} when $\nu\geq 1$. For the difference estimate \eqref{fracderivalphaD}, setting $G_{\alpha}(z,\overline{z})=(z\overline{z})^{\frac{\alpha}{2}}z$, the proof is similar to the deduction of \eqref{fracderivdiff}.    
\end{proof}

\subsection{Proof of Theorem \ref{LWPresults}}

We present the key nonlinear estimates for our local well-posedness result in $H^s$-spaces.
\begin{lemma}\label{nonlinearHS} Let $s>\frac{N}{2}$ and $\alpha\geq \lceil s\rceil$. Assume that $u, v\in H^s(\mathbb{R}^N)$. Then
\begin{equation*}
  \||u|^{\alpha} u\|_{H^s}\lesssim \|u\|_{H^s}^{\alpha+1}  
\end{equation*}
and
\begin{equation*}
\begin{aligned}
 \||u|^{\alpha} u-|v|^{\alpha}v\|_{H^s}\lesssim \big(1+\|u\|_{H^s}+\|v\|_{H^s}\big)^{\alpha}\|u-v\|_{H^s}   
\end{aligned}    
\end{equation*}
\end{lemma}
\begin{proof}
We write $s=s_1+\theta$, where $s_1\in \mathbb{Z}^{+}$ and $\theta\in [0,1)$. We restrict to the case $\theta \in (0,1)$ as the case $\theta=0$ follows by controlling only local derivatives and with simple modifications to our argument below. Thus, our goal is to address the fractional case, which is more involved.  We consider two main cases $\frac12<s<1$ and $s>\max\{1,\frac{N}{2}\}$.

\underline{\bf Regularity $\frac{1}{2}< s <1$}. Since the condition $s>\frac{N}{2}$, it must follow that the dimension $N=1$. Using Theorem \ref{stein}, Lemma \ref{lemmafracderNON}, and Sobolev embedding $H^s(\mathbb{R})\hookrightarrow L^{\infty}(\mathbb{R})$, we get 
\begin{equation}\label{Nonlest1}
\begin{aligned}
 \||u|^{\alpha}u \|_{H^s}\lesssim  &\||u|^{\alpha} u\|_{L^2}+\|\mathcal{D}^s\big(|u|^{\alpha} u\big)\|_{L^2} \\
 \lesssim & \|u\|_{H^s}^{\alpha+1}+\|u\|_{L^{\infty}}\|\mathcal{D}^s(u)\|_{L^2}\\
 \lesssim & \|u\|_{H^s}^{\alpha+1}.
\end{aligned}    
\end{equation}
Using that $\alpha\geq 1$, Theorem \ref{stein}, Lemma \ref{lemmafracderNON} and Sobolev embedding, we get 
\begin{equation}\label{Nonlest2}
\begin{aligned}
 \||u|^{\alpha}u-|v|^{\alpha}v \|_{H^s}\lesssim  &\||u|^{\alpha} u-|v|^{\alpha}v\|_{L^2}+\|\mathcal{D}^s\big(|u|^{\alpha} u-|v|^{\alpha} v\big)\|_{L^2} \\
 \lesssim & \big(\|u\|_{L^{\infty}}^{\alpha}+\|v\|_{L^{\infty}}^{\alpha}\big)\|u-v\|_{L^2}\\
 &+\big(\|u\|_{L^{\infty}}^{\alpha-1}+\|v\|_{L^{\infty}}^{\alpha-1}\big)\big(\|\mathcal{D}^{\theta}(u)\|_{L^2}+\|\mathcal{D}^{\theta}(v)\|_{L^2}\big)\|u-v\|_{L^{\infty}}\\
 \lesssim & \big(\|u\|_{H^s}^{\alpha}+\|v\|_{H^s}^{\alpha}\big)\|u-v\|_{H^s}.
\end{aligned}    
\end{equation}
This completes the estimates in this case.
\smallskip

\underline{\bf Regularity $s>\max\Big\{1,\frac{N}{2}\Big\}$}. Recall that $s=s_1+\theta$, where $s_1\in \mathbb{Z}^{+}$ and $\theta\in (0,1)$. To estimate the $H^s$-norm of $|u|^{\alpha}u$, it suffices to show that $\partial^{\beta}(|u|^{\alpha}u)\in H^{\theta}(\mathbb{R}^N)$ for any $|\beta|=s_1$. Thus, we decompose $\partial^{\beta}(|u|^{\alpha}u)$ into $\mathscrsfs{A}$ and $\mathscrsfs{B}$ as in \eqref{deriv} and split our reasoning accordingly.
\smallskip

\underline{Estimates for $\mathscrsfs{A}$}. After using the equivalence between the Riesz potential $D^{\theta}$ and $\mathcal{D}^{\theta}$, we apply the fractional type Leibniz rule \eqref{Leib} and  \eqref{fracderiv} with $\alpha\geq 1$ to get 
\begin{equation*}
\begin{aligned}
 \||u|^{\alpha}\partial^{\beta} u\|_{H^{\theta}}\lesssim &  \||u|^{\alpha}\partial^{\beta} u\|_{L^2}+ \|\mathcal{D}^{\theta}\big(|u|^{\alpha}\partial^{\beta} u\big)\|_{L^2}\\
 \lesssim &\|u\|_{L^{\infty}}^{\alpha}\|\partial^{\beta} u\|_{L^2}+\||u|^{\alpha} \mathcal{D}^{\theta}\big(\partial^{\beta} u\big)\|_{L^2}+ \|\mathcal{D}^{\theta}\big(|u|^{\alpha}\big)\partial^{\beta} u\|_{L^2}\\
 \lesssim & \|u\|_{L^{\infty}}^{\alpha}\Big(\|\partial^{\beta} u\|_{L^2}+\| \mathcal{D}^{\theta}\big(\partial^{\beta} u\big)\|_{L^2}\Big)+ \|u\|_{L^{\infty}}^{\alpha-1}\|\mathcal{D}^{\theta}(u)\partial^{\beta} u\|_{L^2}.
\end{aligned}    
\end{equation*}
Let us estimate the second term on the right-hand side of the previous inequality. By H\"older's inequality, we write
\begin{equation}\label{Restrict0}
\begin{aligned}
\|\mathcal{D}^{\theta}(u)\partial^{\beta} u\|_{L^2}\lesssim & \|\mathcal{D}^{\theta}(u)\|_{L^{q_1}}\|\partial^{\beta} u\|_{L^{q_2}},
\end{aligned}    
\end{equation}
where
\begin{equation}\label{Restrict1}
\left\{  \begin{aligned}
   &\frac{1}{q_j}=\frac{1}{2}-\frac{\kappa_j}{N}, \quad 0<\kappa_j< \frac{N}{2}, \quad j=1,2, \\
   &\kappa_1+\kappa_2=\frac{N}2{}.
  \end{aligned} \right. 
\end{equation}
Assuming first that $s-|\beta|<\frac{N}{2}$ and that $0<\theta<\min\{1,\frac{N}{2}\}$, we take $\kappa_2=s-|\beta|$, and $\kappa_1=\frac{N}{2}-s+|\beta|$ in \eqref{Restrict1}. We note that $\kappa_1 \leq \frac{N}{2}-\theta$, i.e., $\kappa_1+\theta\leq \frac{N}{2}<s$. Using that $\|\mathcal{D}^{\theta}(u)\|_{L^{q_1}}\lesssim \|J^{\theta}(u)\|_{L^{q_1}}$ (by Theorem \ref{stein}), and Hardy-Littlewood-Sobolev inequality, it follows that
\begin{equation*}
 \begin{aligned}
   \|\mathcal{D}^{\theta}(u)\partial^{\beta} u\|_{L^2}\lesssim & \|\mathcal{D}^{\theta}(u)\|_{L^{q_1}}\|\partial^{\beta} u\|_{L^{q_2}}\lesssim \|J^{\kappa_1+\theta} u\|_{L^2}\|J^s u\|_{L^2}\lesssim \|u\|_{H^s}^2.   
 \end{aligned}   
\end{equation*}

Next, we assume that $s-|\beta|<\frac{N}{2}$ and that $\frac{N}{2}\leq \theta <1$, which forces the dimension $N=1$. If $\theta>\frac{1}{2}$, we use Sobolev embedding $H^{\theta}(\mathbb{R})\hookrightarrow L^{\infty}(\mathbb{R})$ to deduce
\begin{equation*}
\begin{aligned}
 \|\mathcal{D}^{\theta}(u)\partial^{\beta} u\|_{L^2}\lesssim \|\partial^{\beta}u\|_{L^{\infty}}\|\mathcal{D}^{\theta}(u)\|_{L^2}\lesssim \|u\|_{H^s}\|u\|_{H^{\theta}}\lesssim \|u\|_{H^{s}}^2.    
\end{aligned}    
\end{equation*}
If $\theta=\frac{1}{2}$, using that $s>1$ and that $s_1+\frac{1}{4}\leq s$, by Sobolev embedding we get
\begin{equation*}
\begin{aligned}
 \|\mathcal{D}^{\frac{1}{2}}(u)\partial^{\beta} u\|_{L^2}\lesssim \|\partial^{\beta}u\|_{L^{4}}\|\mathcal{D}^{\frac{1}{2}}(u)\|_{L^{4}}\lesssim \|u\|_{H^s}\|u\|_{H^{\frac{3}{2}}}\lesssim \|u\|_{H^{s}}^2.    
\end{aligned}    
\end{equation*}
Next, we assume that $s-|\beta|\geq \frac{N}{2}$ (observe that there is no restriction on the size of $0<\theta<1$). We take any fixed $0<\kappa_1<\min\{s-\theta,\frac{N}{2}\}$, and we set $\kappa_2=\frac{N}{2}-\kappa_1$. Using these restrictions in \eqref{Restrict0} and Hardy-Littlewood-Sobolev inequality
\begin{equation*}
 \begin{aligned}
   \|\mathcal{D}^{\theta}(u)\partial^{\beta} u\|_{L^2}\lesssim \|J^{\kappa_1+\theta} u\|_{L^2}\|J^{|\beta|+\kappa_2} u\|_{L^2}\lesssim \|u\|_{H^s}^2,
 \end{aligned}   
\end{equation*}
where we have used that $\kappa_1+\theta\leq \frac{N}{2}+\theta\leq s-(|\beta|-\theta)<s$, provided that $|\beta|=s_1\geq 1$, and that $|\beta|+\kappa_2\leq |\beta|+\frac{N}{2}\leq s$. 

Summarizing, we obtain
\begin{equation*}
\begin{aligned}
 \||u|^{\alpha}\partial^{\beta} u\|_{H^{\theta}}\lesssim \|u\|^{\alpha+1}_{H^{s}}.   
\end{aligned}    
\end{equation*}

\underline{Estimates for $\mathscrsfs{B}$}. In this case, we need to estimate the fractional derivative of the expressions
\begin{equation*}
\begin{aligned}
\Big\|\mathcal{D}^{\theta}\Big(|u|^{\alpha-2p}\,\,\partial^{\beta-\gamma} u \,\,\partial^{\gamma_{1,1}} u\,\,\partial^{\gamma_{2,1}}\bar{u}\dots\partial^{\gamma_{1,p}} u\,\,\partial^{\gamma_{2,p}} \bar{u}\Big)\Big\|_{L^2},
\end{aligned}    
\end{equation*}
where $1\leq |\gamma|\leq |\beta|$, $1\leq p\leq |\gamma|$, $\gamma_1+\dots+\gamma_p=\gamma$, $|\gamma_j|\geq1$, and $\gamma_{1,j}+\gamma_{2,j}=\gamma_j$, $1\leq j\leq p$.

We first claim that there exists an integer $0 \leq r \leq p$ such that
\begin{equation*}
 \alpha + r - 2p \geq 1,
\end{equation*}
and among the factors
\begin{equation*}
\partial^{\beta-\gamma} u \,\,\partial^{\gamma_{1,1}} u\,\,\partial^{\gamma_{2,1}}\bar{u}\dots\partial^{\gamma_{1,p}} u\,\,\partial^{\gamma_{2,p}} \bar{u}    
\end{equation*}
at least $r$ of them involve derivatives of order zero. If such an $r$ does not exist, then necessarily $\alpha < p+1 \leq s_1+1$, where we recall that $s = s_1 + \theta$ with $s_1 \in \mathbb{Z}^+$. But this contradicts our assumption $\lfloor s \rfloor = s_1+1 \leq \alpha$. Consequently, we write
\begin{equation*}
 |u|^{\alpha-2p}\,\,\partial^{\beta-\gamma} u \,\,\partial^{\gamma_{1,1}} u\,\,\partial^{\gamma_{2,1}}\bar{u}\dots\partial^{\gamma_{1,p}} u\,\,\partial^{\gamma_{2,p}} \bar{u}= |u|^{\alpha-2p}\prod_{j=1}^{r}w_j \bigg(\prod_{l=1}^{2p+1-r} \partial^{\widetilde{\gamma}_l}w_l\bigg),  
\end{equation*}
where $w_j$ and $w_l$ are equal to either $u$ or $\overline{u}$, and $\sum_{l=1}^{2p+1-r} |\widetilde{\gamma}_l|=|\beta|$, and we assume that the empty sum is equal to zero. Applying the Leibniz rule \eqref{Lei}, \eqref{fracderiv} in Lemma \ref{GlemmafracderNON}, we get
\begin{equation*}
\begin{aligned}
\Big\|\mathcal{D}^{\theta} &\Big(|u|^{\alpha-2p}\,\,\partial^{\beta-\gamma} u \,\,\partial^{\gamma_{1,1}} u\,\,\partial^{\gamma_{2,1}}\bar{u}\dots\partial^{\gamma_{1,p}} u\,\,\partial^{\gamma_{2,p}} \bar{u}\Big)\Big\|_{L^2}\\
\lesssim & \bigg\|\mathcal{D}^{\theta}\bigg(|u|^{\alpha-2p}\prod_{j=1}^{r}w_j \bigg)\bigg(\prod_{l=1}^{2p+1-r} \partial^{\widetilde{\gamma}_l}u\bigg)\bigg\|_{L^2}+\bigg\||u|^{\alpha+r-2p}\, \mathcal{D}^{\theta}\bigg(\prod_{l=1}^{2p+1-r} \partial^{\widetilde{\gamma}_l}u\bigg)\bigg\|_{L^2}\\
\lesssim & \|u\|_{L^{\infty}}^{\alpha+r-2p-1}\bigg\|\mathcal{D}^{\theta}(u)\bigg(\prod_{l=1}^{2p+1-r} \partial^{\widetilde{\gamma}_l}u\bigg)\bigg\|_{L^2}+\|u\|_{L^{\infty}}^{\alpha+r-2p}\bigg\|D^{\theta}\bigg(\prod_{l=1}^{2p+1-r} \partial^{\widetilde{\gamma}_l}u\bigg)\bigg\|_{L^2}\\
\lesssim&~\mathcal{I}_{1}+\mathcal{I}_{2}.
\end{aligned}    
\end{equation*}
To estimate $\mathcal{I}_1$, we proceed as in $\mathscrsfs{A}$ above. Let us consider some simple cases. If any of the derivatives $\widetilde{\gamma}_l=\beta$, then $\mathcal{I}_1$ reduces to
\begin{equation}\label{casehihgerorderder}
 \begin{aligned}
\mathcal{I}_1\lesssim \|u\|_{L^{\infty}}^{\alpha+r-2p-1}\bigg\|\mathcal{D}^{\theta}(u)\partial^{\beta}u\bigg(\prod_{l=1}^{2p-r}u\bigg)\bigg\|_{L^2}\lesssim \|u\|_{L^{\infty}}^{\alpha-1}\|\mathcal{D}^{\theta}(u)\, \partial^{\beta}u\|_{L^2},
 \end{aligned}   
\end{equation}
whose estimate is the same as obtained for $\mathscrsfs{A}$ above. Hence, we will assume that $|\widetilde{\gamma}_l|<|\beta|$ for all $l=1,\dots, 2p+1-r$.

Next, we consider the cases $\frac{N}{2}\leq \theta <1$, which forces $N=1$. If $\theta>\frac{1}{2}$ in dimension $N=1$, we use Sobolev embedding $H^{\theta}(\mathbb{R})\hookrightarrow L^{\infty}(\mathbb{R})$ to get
\begin{equation*}
\begin{aligned}
\mathcal{I}_1\lesssim   \|u\|_{L^{\infty}}^{\alpha+r-2p-1}\bigg(\prod_{l=1}^{2p+1-r}\|\partial^{\widetilde{\gamma}_{l}}u\|_{L^{\infty}}\bigg)\|\mathcal{D}^{\theta}(u)\|_{L^2}  \lesssim \|u\|^{\alpha+1}_{H^{s}}.   
\end{aligned}  
\end{equation*}
If $\theta=\frac{1}{2}$ in dimension $N=1$, consider $\frac{1}{q}=\frac{1}{2}-\frac{2p+1-r}{2(2p+2-r)}$, and note that $0<\theta_1:=\frac{2p+1-r}{2(2p+2-r)}<\frac{1}{2}$. Using H\"older's inequality and Sobolev embedding $H^{\theta_1}(\mathbb{R})\hookrightarrow L^q(\mathbb{R})$, we deduce 
\begin{equation*}
\begin{aligned}
\mathcal{I}_1\lesssim &  \|u\|_{L^{\infty}}^{\alpha+r-2p-1}\bigg(\prod_{l=1}^{2p+1-r}\|\partial^{\widetilde{\gamma}_{l}}u\|_{L^{q}}\bigg)\|\mathcal{D}^{\frac{1}{2}}(u)\|_{L^q}\\
\lesssim & \|u\|_{H^s}^{\alpha}\|J^{\frac{1}{2}} u\|_{H^{\theta_1}}  \lesssim \|u\|^{\alpha+1}_{H^{s}},   
\end{aligned}  
\end{equation*}
where we have used that $\|\mathcal{D}^{\frac{1}{2}}(u)\|_{L^q}\lesssim \|J^{\frac{1}{2}} u\|_{L^q}$, which is a consequence of Theorem \ref{stein}, and that $\frac{1}{2}+\theta_1\leq s$. 

Consequently, from now on, we assume that $0<\theta<\frac{N}{2}$. By H\"older's inequality, we have
\begin{equation*}
 \begin{aligned}
   \mathcal{I}_1\lesssim \|u\|_{L^{\infty}}^{\alpha+r-2p-1}\|\mathcal{D}^{\theta}(u)\|_{L^{q_0}}\prod_{l=1}^{2p+1-r}\|\partial^{\widetilde{\gamma_l}} u\|_{L^{q_l}},
 \end{aligned}   
\end{equation*}
where
\begin{equation}\label{genrestricitons}
\left\{\begin{aligned}
&\frac{1}{q_l}=\frac{1}{2}-\frac{\kappa_l}{N}, \quad 0<\kappa_l<\frac{N}{2}, \quad l=0,\dots,2p+1-r, \\
&\kappa_0+\sum_{l=1}^{2p+1-r} \kappa_j=\frac{N}{2}(2p+1-r).
 \end{aligned}   \right.
\end{equation}
Note that $\sum\limits_{l=0}^{2p+1-r}\frac{1}{q_l}=\frac{1}{2}$. Motivated by the estimate of the case $\mathscrsfs{A}$ above, we consider two sets
\begin{equation}\label{Asettoparti}
\begin{aligned}
&\mathcal{A}_{1}=\big\{l\in \{1,2,\dots, 2p+1-r\} : s-|\widetilde{\gamma_l}|<\frac{N}{2}\big\}, \\
&\mathcal{A}_{2}=\big\{l\in \{1,2,\dots, 2p+1-r\} : s-|\widetilde{\gamma_l}|\geq \frac{N}{2}\big\},
\end{aligned}   
\end{equation}
and we set
\begin{equation*}
 \begin{aligned}
& \kappa_l=s-|\widetilde{\gamma}_l|, \, \, \text{ if } \, \, l\in \mathcal{A}_1, \\
& \kappa_l=\frac{N}{2}-\epsilon, \, \, \text{ if } \, \, l\in \mathcal{A}_2,
 \end{aligned}   
\end{equation*}
where if $\mathcal{A}_2\neq \emptyset$, $\epsilon>0$ is such that 
\begin{equation}\label{epsiEQ}
 \epsilon |\mathcal{A}_2|\leq \frac{N}{2}-\theta+\sum_{l\in\mathcal{A}_1}\big(s-\frac{N}{2}-|\widetilde{\gamma}_l|\big).  
\end{equation}
Here, $|\mathcal{A}_j|$ denotes the number of elements in the set $\mathcal{A}_j$, $j=1,2$. Thus, it follows that
\begin{equation*}
\begin{aligned}
\kappa_0=\frac{N}{2}(2p+1-r)-\sum_{l\in \mathcal{A}_1}\big(s-|\widetilde{\gamma}_l|\big)-\sum_{l\in \mathcal{A}_2}\big(\frac{N}{2}-\epsilon\big). 
\end{aligned}
\end{equation*}
It is worth recalling that we are assuming $|\widetilde{\gamma}_l|<|\beta|$, since the case $|\widetilde{\gamma}_l|=|\beta|$ for some $l$ was treated in \eqref{casehihgerorderder}. Also recall that since $s>0$ is fractional, $0<\kappa_l<\frac{N}{2}$ for all $l\in \mathcal{A}_1$. We require the following result.
\begin{claim}\label{positvclaim}
Under the assumption that $\sum_{l=1}^{2p+1-r} |\widetilde{\gamma}_l|=|\beta|$  and that $|\widetilde{\gamma}_l|<|\beta|$, for each $l=1,\dots, 2p+1-r$, it follows    
\begin{equation}\label{postineq}
  \frac{N}{2}-\theta+\sum_{l\in\mathcal{A}_1}\big(s-\frac{N}{2}-|\widetilde{\gamma}_l|\big)>0 . 
\end{equation}
Moreover, if $\mathcal{A}_2 \neq \emptyset$, choose $\epsilon>0$ such that \eqref{epsiEQ} holds. Regardless of whether $\mathcal{A}_1$ is empty or not, it follows that $0 < \kappa_0 \leq \tfrac{N}{2} - \theta$. 
\end{claim}

We remark that \eqref{postineq} implies the existence of $\epsilon>0$, for which \eqref{epsiEQ} hold.
\begin{proof}[Proof of Claim \ref{positvclaim}]
We first deduce \eqref{postineq}. If $\mathcal{A}_1=\emptyset$, the inequality above is $\frac{N}{2}-\theta>0$. If $\mathcal{A}_1\neq \emptyset$, we write
\begin{equation*}
\begin{aligned}
\frac{N}{2}-\theta+\sum_{l\in\mathcal{A}_1}\big(s-\frac{N}{2}-|\widetilde{\gamma}_l|\big)=s-\theta-\sum_{l\in \mathcal{A}_1}|\widetilde{\gamma}_i|+\Big(s-\frac{N}{2}\Big)\big(|\mathcal{A}_1|-1\big).    
\end{aligned}    
\end{equation*}
If $|\mathcal{A}_1|>1$, using that $s-\theta-\sum_{l\in \mathcal{A}_i}|\widetilde{\gamma}_i|\geq 0$ and $s>\frac{N}{2}$, the inequality above implies \eqref{postineq}. Now, if $|\mathcal{A}_1|=1$, let $i\in \mathcal{A}_1$. Using that $|\widetilde{\gamma}_i|<\beta$, we deduce
\begin{equation*}
 \begin{aligned}
s-\theta-\sum_{l\in \mathcal{A}_1}|\widetilde{\gamma}_i|+\Big(s-\frac{N}{2}\Big)\big(|\mathcal{A}_1|-1\big)=s-\theta-|\widetilde{\gamma}_i|>s-\theta-|\beta|\geq 0.
 \end{aligned}   
\end{equation*}
The proof of \eqref{postineq} is complete.

On the other hand, let us deduce that $\kappa_0\leq \frac{N}{2}-\theta$. Using that $\frac{N}{2}(2p+1-r)=\sum_{l=1}^{2p+1-r} \frac{N}{2}=\sum_{l\in \mathcal{A}_1}\frac{N}{2}+\sum_{l\in \mathcal{A}_2}\frac{N}{2},$ we have that $\kappa_0\leq \frac{N}{2}-\theta$ is equivalent to
\begin{equation*}
 \begin{aligned}
0\leq &  \frac{N}{2}-\theta-\frac{N}{2}(2p+1-r)+\sum_{l\in \mathcal{A}_1}\big(s-|\widetilde{\gamma}_l|\big)+\sum_{l\in \mathcal{A}_2}\big(\frac{N}{2}-\epsilon\big) \\
=&  \frac{N}{2}-\theta+\sum_{l\in \mathcal{A}_1}\Big(s-\frac{N}{2}-|\widetilde{\gamma}_l|\Big)-\epsilon|\mathcal{A}_2|,
 \end{aligned}   
\end{equation*}
which by \eqref{postineq} is always true when $\mathcal{A}_2=\emptyset$, and when $\mathcal{A}_2\neq \emptyset$, it holds from our choice of $\epsilon>0$. To check that $\kappa_0>0$, we use $s-|\gamma_l|<\frac{N}{2}$, for all $l\in \mathcal{A}_1$ and that $\frac{N}{2}-\epsilon<\frac{N}{2}$ to obtain
\begin{equation*}
\begin{aligned}
\kappa_0=&\frac{N}{2}(2p+1-r)-\sum_{l\in \mathcal{A}_1}\big(s-|\widetilde{\gamma}_l|\big)-\sum_{l\in \mathcal{A}_2}\big(\frac{N}{2}-\epsilon\big)\\
>& \frac{N}{2}(2p+1-r)-\frac{N}{2}\big(|\mathcal{A}_1|+|\mathcal{A}_2|\big)=0,
\end{aligned}    
\end{equation*}
where we used that $|\mathcal{A}_1|+|\mathcal{A}_2|=2p+1-r$. We conclude that $0<\kappa_0<\frac{N}{2}-\theta$.
\end{proof}

Consequently, the previous theorem implies that we can apply Hardy-Littlewood-Sobolev inequality with restrictions \eqref{genrestricitons} to deduce that $H^{\kappa_l}(\mathbb{R}^N)\hookrightarrow L^{q_{l}}(\mathbb{R}^N)$, $l=1,\dots, 2p_1-r$. Moreover, if $l\in \mathcal{A}_1$, $|\widetilde{\gamma}_{l}|+\kappa_l\leq s$, and if $l\in \mathcal{A}_2$, $|\widetilde{\gamma}_{l}|+\kappa_l= |\widetilde{\gamma}_l|+\frac{N}{2}-\epsilon\leq s-\epsilon<s$. We also note that $\theta+\kappa_0\leq \frac{N}{2}$. Thus, using Sobolev embedding $H^s(\mathbb{R}^N)\hookrightarrow L^{\infty}(\mathbb{R}^N)$, we deduce
\begin{equation*}
\begin{aligned}
\mathcal{I}_1\lesssim & \|u\|_{L^{\infty}}^{\alpha+r-2p-1}\|J^{\theta}(u)\|_{L^{q_0}}\prod_{l=1}^{2p+1-r}\|\partial^{\widetilde{\gamma_l}} u\|_{L^{q_l}}\\
\lesssim & \|u\|_{L^{\infty}}^{\alpha+r-2p-1}\|J^{\theta}u\|_{H^{\kappa_0}}\prod_{l=1}^{2p+1-r}\|\partial^{\widetilde{\gamma_l}} u\|_{H^{\kappa_l}}\\
\lesssim & \|u\|_{H^s}^{\alpha}\|J^{\frac{N}{2}}u\|_{L^2}\\
\lesssim & \|u\|_{H^s}^{\alpha+1}.
 \end{aligned}   
\end{equation*}
Gathering the above results, we complete the estimate of $\mathcal{I}_1$. We now focus on $\mathcal{I}_2$. We first reduce our considerations to the case $0<\theta<\frac{N}{2}$. In dimension $N=1$, if $\frac{1}{2}<\theta<1$, we apply  Lemma \ref{Grafakos-Oh} and Sobolev embedding $H^{\theta}(\mathbb{R})\hookrightarrow L^{\infty}(\mathbb{R})$ to deduce
\begin{equation*}
 \begin{aligned}
 \mathcal{I}_2 \lesssim  \|u\|_{L^{\infty}}^{\alpha+r-2p}\sum_{m=1}^{2p+1-r}\|D^{\theta}(\partial^{\widetilde{\gamma}_m} u)\|_{L^{2}}\prod_{\substack{1\leq l\leq 2p+1-r\\ l\neq m}} \|\partial^{\widetilde{\gamma}_l} u\|_{L^{\infty}}\lesssim \|u\|_{H^s}^{\alpha+1}.
 \end{aligned}   
\end{equation*}
If $\theta=\frac{1}{2}$ and $\theta+|\widetilde{\gamma}_m|=s$, we argue exactly as before, i.e., taking the $L^2$-norm of $D^{\theta}(\partial^{\widetilde{\gamma}_m} u)$, and the $L^{\infty}$-norm of the remaining factors.  
Therefore, we assume that $|\widetilde{\gamma}_m|\leq s-1$. Then consider $\frac{1}{q}=\frac{1}{2}-\frac{2p-r}{2(2p+1-r)}$ and observe that $0<\theta_1:=\frac{2p-r}{2(2p+1-r)}<\frac{1}{2}$. Using Lemma \ref{Grafakos-Oh} and Sobolev embedding $H^{\theta_1}(\mathbb{R})\hookrightarrow L^q(\mathbb{R})$, we get 
\begin{equation*}
 \begin{aligned}
 \mathcal{I}_2 \lesssim  \|u\|_{L^{\infty}}^{\alpha+r-2p}\sum_{m=1}^{2p+1-r}\|D^{\frac{1}{2}}(\partial^{\widetilde{\gamma}_m} u)\|_{L^{q}}\prod_{\substack{1\leq l\leq 2p+1-r\\ l\neq m}} \|\partial^{\widetilde{\gamma}_l} u\|_{L^{q}}\lesssim \|u\|_{H^s}^{\alpha+1}.
 \end{aligned}   
\end{equation*}
Consequently, we will assume that $0<\theta<\frac{N}{2}$. The previous estimate reveals that if $\theta+|\widetilde{\gamma}_n|=s$, the desired result follows from taking the $L^2$-norm of $D^{\theta}(\partial^{\widetilde{\gamma}_m} u)$, and the $L^{\infty}$-norm of the other terms. Thus, we assume that $0<s-|\widetilde{\gamma}_m|-\theta$. We also mention that the case $|\widetilde{\gamma}_l|=|\beta|$ for some $l=1,\dots, 2p+1-r$ can be treated similarly. Thus, we consider $\sum_{l=1}^{2p+1-r} |\widetilde{\gamma}_l|=|\beta|$  with $|\widetilde{\gamma}_l|<|\beta|$, for each $l=1,\dots, 2p+1-r$.

By several applications of the fractional Leibniz rule in Lemma \ref{Grafakos-Oh} and H\"older's inequality, we get
\begin{equation*}
 \begin{aligned}
 \mathcal{I}_2 \lesssim  \|u\|_{L^{\infty}}^{\alpha+r-2p} \sum_{m=1}^{2p+1-r}\|D^{\theta}(\partial^{\widetilde{\gamma}_m} u)\|_{L^{q_m}}\prod_{\substack{1\leq l\leq 2p+1-r\\ l\neq m}} \|\partial^{\widetilde{\gamma}_l} u\|_{L^{q_l}},
 \end{aligned}   
\end{equation*}
where following the ideas in the estimate for $\mathcal{I}_1$ above, we define 
\begin{equation}
\left\{\begin{aligned}
&\frac{1}{q_l}=\frac{1}{2}-\frac{\kappa_l}{N}, \quad 0<\kappa_l<\frac{N}{2}, \quad l=1,\dots,2p+1-r, \\
&\sum_{l=1}^{2p+1-r} \kappa_j=\frac{N}{2}(2p-r).
 \end{aligned}   \right.
\end{equation}
We consider the same sets $\mathcal{A}_1$ and $\mathcal{A}_2$ in \eqref{Asettoparti} to define
\begin{equation*}
 \begin{aligned}
& \kappa_l=s-|\widetilde{\gamma}_l|, \, \, \text{ if } \, \, l\in \mathcal{A}_1, \, l\neq m, \\
& \kappa_l=\frac{N}{2}-\epsilon, \, \, \text{ if } \, \, l\in \mathcal{A}_2, \, l\neq m,
 \end{aligned}   
\end{equation*}
where if $\mathcal{A}_2\neq \emptyset$, we choose $0<\epsilon<\frac{N}{2}$ later. Hence, we have
\begin{equation}\label{kappacond}
 \kappa_m =\frac{N}{2}(2p-r)-\sum_{\substack{l\in \mathcal{A}_1\\ l\neq m}} \big(s-|\widetilde{\gamma}_l|\big)-\sum_{\substack{l\in \mathcal{A}_2\\ l\neq m}} \big(\frac{N}{2}-\epsilon\big). 
\end{equation}
We remark that the argument at the end of the proof of Claim \ref{positvclaim} shows $\kappa_m>0$. Next, we consider two cases depending on $s-|\widetilde{\gamma}_m|$.

\underline{\bf Case: $0<s-|\widetilde{\gamma}_m|<\frac{N}{2}$}.  Let us check under which conditions $0<\kappa_m\leq s-|\widetilde{\gamma}_m|-\theta$. Using that $\frac{N}{2}(2p-r)=\frac{N}{2}\Big(\sum_{\substack{l\in \mathcal{A}_1\\ l\neq m}} +\sum_{\substack{l\in \mathcal{A}_2}}\Big)$, it follows from \eqref{kappacond} that $\kappa_m\leq s-|\widetilde{\gamma}_m|-\theta$ is equivalent to
\begin{equation*}
 \begin{aligned}
     \epsilon|\mathcal{A}_2|\leq &s-|\widetilde{\gamma}_m|-\theta+\sum_{\substack{l\in \mathcal{A}_1\\ l\neq m}} \big(s-\frac{N}{2}+|\widetilde{\gamma}_l|\big)
     = \frac{N}{2}-\theta +\sum_{l\in \mathcal{A}_1} \big(s-\frac{N}{2}-|\widetilde{\gamma}_l|\big),
 \end{aligned}   
\end{equation*}
where the right-hand side of the above inequality is positive by Claim \ref{positvclaim}. As a consequence, if $|\mathcal{A}_2|=0$, it is always the case $0<\kappa_m\leq s-|\widetilde{\gamma}_m|-\theta$. If $|\mathcal{A}_2|\neq 0$, to get the same conclusion, it is enough to take $0<\epsilon<\frac{N}{2}$ such that
\begin{equation*}
\begin{aligned}
 \epsilon |\mathcal{A}_2| \leq \frac{N}{2}-\theta +\sum_{l\in \mathcal{A}_1} \big(s-\frac{N}{2}-|\widetilde{\gamma}_l|\big).
\end{aligned}    
\end{equation*}
Hence, under the previous restriction, we apply the Sobolev embedding $H^{\kappa_l}(\mathbb{R}^N)\hookrightarrow L^{q_l}(\mathbb{R}^N)$ to deduce
\begin{equation*}
 \begin{aligned}
 \|D^{\theta}(\partial^{\widetilde{\gamma}_m} u)\|_{L^{q_m}}\prod_{\substack{1\leq l\leq 2p+1-r\\ l\neq m}} \|\partial^{\widetilde{\gamma}_l} u\|_{L^{q_l}}\lesssim \|u\|_{H^s}^{2p+1-r}.
 \end{aligned}   
\end{equation*}
\underline{\bf Case: $s-|\widetilde{\gamma}_m|\geq \frac{N}{2}$}. Here, we claim that $\kappa_{m}\leq \frac{N}{2}-\theta$. To see this, we write $\frac{N}{2}(2p-r)=\frac{N}{2}\Big(\sum_{\substack{l\in \mathcal{A}_1}} +\sum_{\substack{l\in \mathcal{A}_2\\ l\neq m}}\Big)$ to deduce that $\kappa_{m}\leq \frac{N}{2}-\theta$ is equivalent to
\begin{equation*}
\epsilon \big(|\mathcal{A}_2|-1\big)\leq \frac{N}{2}-\theta +\sum_{l\in \mathcal{A}_1} \big(s-\frac{N}{2}-|\widetilde{\gamma}_l|\big),    
\end{equation*}
where it is worth emphasizing that, since $s-|\widetilde{\gamma}_m|\geq \frac{N}{2}$, $|\mathcal{A}_2|\geq 1$. Consequently, if $|\mathcal{A}_2|=1$, it is always the case $\kappa_m\leq \frac{N}{2}-\theta$. If $|\mathcal{A}_2|\geq 2$, we can take $0<\epsilon<\frac{N}{2}$ such that the above inequality holds. Thus, the desired result is now a consequence of the embedding $H^{\kappa_l}(\mathbb{R}^N)\hookrightarrow L^{q_l}(\mathbb{R}^N)$, $l=1,\dots, 2p+1-r$. For example, one has $\|D^{\theta}(\partial^{\widetilde{\gamma}_m} u)\|_{L^{q_m}}\lesssim  \|J^{|\widetilde{\gamma}_m|+\theta+\kappa_m} u\|_{L^{2}}\lesssim \|J^{|\widetilde{\gamma}_m|+\frac{N}{2}} u\|_{L^{2}}\lesssim \|u\|_{H^s}$.

Summarizing, we conclude that
\begin{equation*}
 \begin{aligned}
  \mathcal{I}_2\lesssim \|u\|_{H^s}^{\alpha+1},   
 \end{aligned}   
\end{equation*}
which completes the proof of the estimate for $|u|^{\alpha} u\in H^s(\mathbb{R})$. 

We remark that in the study of the previous estimates, we have identified the relevant indices to apply different Sobolev embeddings according to the case at hand. As a consequence,  when dealing with the difference $|u|^{\alpha}u-|v|^{\alpha}v$, the idea is to group the terms as in the proof of Lemma \ref{nonlinear est}, and use the result of Lemma \ref{GlemmafracderNON} together with the previous restrictions to obtain the same conclusion. To avoid unnecessary repetitions, we summarize this estimate as 
\begin{equation*}
 \||u|^{\alpha}u-|v|^{\alpha}v\|_{H^s}\lesssim \big(1+\|u\|_{H^{\alpha}}+\|v\|_{H^{\alpha}}\big)^{\alpha}\|u-v\|_{H^s}, 
\end{equation*}
completing the proof of the lemma. 
\end{proof}

\begin{proof}[Proof of Theorem \ref{LWPresults}]
Given the validity of the nonlinear estimates in Lemma \ref{nonlinearHS}, the existence of solutions follows from a fixed-point argument applied to the integral formulation of \eqref{biNLS}, similar to the proof of Theorem \ref{theorem1}. We omit this and further standard details concerning the local well-posedness.
\end{proof}

\bigskip






{\bf ORCID}

{\it Iryna Petrenko} 
\qquad 
{https://orchid.org/
0009-0008-6579-7117}

{\it Oscar Ria\~no}
\qquad {https://orcid.org/0000-0002-6325-8848}

{\it Svetlana Roudenko}
\qquad 
{https://orcid.org/0000-0002-7407-7639}


\bibliographystyle{abbrv}
\bibliography{references}

@article{Aceves2023,
    AUTHOR = {Aceves, Alejandro},
     TITLE = {Discrete optical solitons: Perspectives and new trends},
   JOURNAL = {Optics Commun.},
  FJOURNAL = {Optics Communications},
    VOLUME = {545},
      YEAR = {2023},
     PAGES = {129713},
      ISSN = {0030-4018},
       DOI = {10.1016/j.optcom.2023.129713},
}

@article{PA2024,
    AUTHOR = {Parker, Ross and Aceves, Alejandro},
     TITLE = {Topological photonics: a mathematical perspective},
   JOURNAL = {Notices AMS},
  FJOURNAL = {Notices of the American Mathematical Society},
    VOLUME = {71},
      YEAR = {2024},
    NUMBER = {8},
     PAGES = {994--1003},
      ISSN = {0002-9920,1088-9477},
   MRCLASS = {78A02 (81V80)},
  MRNUMBER = {4785129},
}

@article{HA2024,
    AUTHOR = {Hetzel, Sabrina and Aceves, Alejandro},
     TITLE = {Generation and interaction of quartic soliton-like pulses in
              an optical cavity},
   JOURNAL = {Phys. D},
  FJOURNAL = {Physica D. Nonlinear Phenomena},
    VOLUME = {481},
      YEAR = {2025},
     PAGES = {Paper No. 134829},
      ISSN = {0167-2789,1872-8022},
   MRCLASS = {99-06},
  MRNUMBER = {4938189},
       DOI = {10.1016/j.physd.2025.134829},
       URL = {https://doi.org/10.1016/j.physd.2025.134829},
}

@article{LightBullets2024,
  title = {Pure quartic three-dimensional spatiotemporal {K}err solitons in graded-index media},
  author = {Parra-Rivas, Pedro and Sun, Yifan and Mangini, Fabio and Ferraro, Mario and Zitelli, Mario and Wabnitz, Stefan},
  journal = {Phys. Rev. A},
  volume = {109},
  issue = {3},
  pages = {033516},
  numpages = {6},
  year = {2024},
  month = {Mar},
  publisher = {American Physical Society},
  doi = {10.1103/PhysRevA.109.033516},
  url = {https://link.aps.org/doi/10.1103/PhysRevA.109.033516}
}

@article{Pastor2023,
    AUTHOR = {Mu\~{n}oz, Alexander and Pastor, Ademir},
     TITLE = {Local well-posedness in weighted {S}obolev spaces for
              nonlinear dispersive equations with applications to dispersive
              blow up},
   JOURNAL = {Math. Ann.},
  FJOURNAL = {Mathematische Annalen},
    VOLUME = {386},
      YEAR = {2023},
    NUMBER = {1-2},
     PAGES = {207--246},
      ISSN = {0025-5831,1432-1807},
   MRCLASS = {35B44 (35Q53)},
  MRNUMBER = {4585148},
MRREVIEWER = {Luiz\ Gustavo\ Farah},
       DOI = {10.1007/s00208-022-02396-7},
       URL = {https://doi.org/10.1007/s00208-022-02396-7},
}

@article {LinaresMiyazaki2019,
    AUTHOR = {Linares, Felipe and Miyazaki, Hayato and Ponce, Gustavo},
     TITLE = {On a class of solutions to the generalized {K}d{V} type
              equation},
   JOURNAL = {Commun. Contemp. Math.},
  FJOURNAL = {Communications in Contemporary Mathematics},
    VOLUME = {21},
      YEAR = {2019},
    NUMBER = {7},
     PAGES = {1850056, 21},
      ISSN = {0219-1997,1793-6683},
   MRCLASS = {35Q53 (35A01 35A21 35B30)},
  MRNUMBER = {4017781},
MRREVIEWER = {Anthony\ D.\ Osborne},
       DOI = {10.1142/S0219199718500566},
       URL = {https://doi.org/10.1142/S0219199718500566},
}

@article {Diana2023,
    AUTHOR = {Friedman, Isaac and Ria\~{n}o, Oscar and Roudenko, Svetlana
              and Son, Diana and Yang, Kai},
     TITLE = {Well-posedness and dynamics of solutions to the generalized
              {K}d{V} with low power nonlinearity},
   JOURNAL = {Nonlinearity},
  FJOURNAL = {Nonlinearity},
    VOLUME = {36},
      YEAR = {2023},
    NUMBER = {1},
     PAGES = {584--635},
      ISSN = {0951-7715,1361-6544},
   MRCLASS = {35Q53},
  MRNUMBER = {4533334},
MRREVIEWER = {Wenbin\ Zhang},
       DOI = {10.1088/1361-6544/ac93e1},
       URL = {https://doi.org/10.1088/1361-6544/ac93e1},
}

@article{ARR2022,
    AUTHOR = {Arora, Anudeep K. and Ria\~{n}o, Oscar and Roudenko, Svetlana},
     TITLE = {Well-posedness in weighted spaces for the generalized
              {H}artree equation with {$p<2$}},
   JOURNAL = {Commun. Contemp. Math.},
  FJOURNAL = {Communications in Contemporary Mathematics},
    VOLUME = {24},
      YEAR = {2022},
    NUMBER = {9},
     PAGES = {Paper No. 2150074, 51},
      ISSN = {0219-1997},
   MRCLASS = {35Q55 (35A01 35B40 42B37)},
  MRNUMBER = {4502388},
       DOI = {10.1142/S0219199721500747},
       URL = {https://doi.org/10.1142/S0219199721500747},
}

@article{HHW2006,
    AUTHOR = {Hao, Chengchun and Hsiao, Ling and Wang, Baoxiang},
     TITLE = {Wellposedness for the fourth order nonlinear {S}chr\"{o}dinger
              equations},
   JOURNAL = {J. Math. Anal. Appl.},
  FJOURNAL = {Journal of Mathematical Analysis and Applications},
    VOLUME = {320},
      YEAR = {2006},
    NUMBER = {1},
     PAGES = {246--265},
      ISSN = {0022-247X},
   MRCLASS = {35Q55 (35B30)},
  MRNUMBER = {2230468},
MRREVIEWER = {Masahito Ohta},
       DOI = {10.1016/j.jmaa.2005.06.091},
       URL = {https://doi.org/10.1016/j.jmaa.2005.06.091},
}

@article{HHW2007,
    AUTHOR = {Hao, Chengchun and Hsiao, Ling and Wang, Baoxiang},
     TITLE = {Well-posedness of {C}auchy problem for the fourth order
              nonlinear {S}chr\"{o}dinger equations in multi-dimensional spaces},
   JOURNAL = {J. Math. Anal. Appl.},
  FJOURNAL = {Journal of Mathematical Analysis and Applications},
    VOLUME = {328},
      YEAR = {2007},
    NUMBER = {1},
     PAGES = {58--83},
      ISSN = {0022-247X},
   MRCLASS = {35Q55 (35B30)},
  MRNUMBER = {2285532},
MRREVIEWER = {Justin A. Holmer},
       DOI = {10.1016/j.jmaa.2006.05.031},
       URL = {https://doi.org/10.1016/j.jmaa.2006.05.031},
}

@article{HMN2016,
    AUTHOR = {Hayashi, Nakao and Mendez-Navarro, Jesus A. and Naumkin, Pavel
              I.},
     TITLE = {Scattering of solutions to the fourth-order nonlinear
              {S}chr\"{o}dinger equation},
   JOURNAL = {Commun. Contemp. Math.},
  FJOURNAL = {Communications in Contemporary Mathematics},
    VOLUME = {18},
      YEAR = {2016},
    NUMBER = {3},
     PAGES = {1550035, 24},
      ISSN = {0219-1997},
   MRCLASS = {35Q55 (35B40 35P25)},
  MRNUMBER = {3477398},
       DOI = {10.1142/S0219199715500352},
       URL = {https://doi.org/10.1142/S0219199715500352},
}

@article {Naum2022,
    AUTHOR = {Naumkin, Pavel I.},
     TITLE = {Asymptotics for the fourth-order nonlinear {S}chr\"odinger
              equation in 2{D}},
   JOURNAL = {Commun. Contemp. Math.},
  FJOURNAL = {Communications in Contemporary Mathematics},
    VOLUME = {24},
      YEAR = {2022},
    NUMBER = {1},
     PAGES = {Paper No. 2050090, 29},
      ISSN = {0219-1997,1793-6683},
   MRCLASS = {35B40 (35Q55 35S30)},
  MRNUMBER = {4378150},
       DOI = {10.1142/S021919972050090X},
       URL = {https://doi.org/10.1142/S021919972050090X},
}

@article{KPV1991,
    AUTHOR = {Kenig, Carlos E. and Ponce, Gustavo and Vega, Luis},
     TITLE = {Oscillatory integrals and regularity of dispersive equations},
   JOURNAL = {Indiana Univ. Math. J.},
  FJOURNAL = {Indiana University Mathematics Journal},
    VOLUME = {40},
      YEAR = {1991},
    NUMBER = {1},
     PAGES = {33--69},
       DOI = {10.1512/iumj.1991.40.40003},
}

@article{BAKS2000,
    AUTHOR = {Ben-Artzi, Matania and Koch, Herbert and Saut, Jean-Claude},
     TITLE = {Dispersion estimates for fourth order {S}chr\"{o}dinger equations},
   JOURNAL = {C. R. Acad. Sci. Paris S\'{e}r. I Math.},
  FJOURNAL = {Comptes Rendus de l'Acad\'{e}mie des Sciences. S\'{e}rie I. Math\'{e}matique},
    VOLUME = {330},
      YEAR = {2000},
    NUMBER = {2},
     PAGES = {87--92},
       DOI = {10.1016/S0764-4442(00)00120-8},
}

@article{K1991,
    AUTHOR = {Karpman, V. I.},
     TITLE = {Influence of high-order dispersion on self-focusing {I}. {Q}ualitative investigation},
   JOURNAL = {Phys. Lett. A},
  FJOURNAL = {Physics Letters. A},
    VOLUME = {160},
      YEAR = {1991},
    NUMBER = {6},
     PAGES = {531--537},
}

@article{CazNaum2016,
author = {Cazenave, Thierry and Naumkin, Ivan},
title = {{Local existence, global existence, and scattering for the nonlinear {S}chr\"odinger equation}},
JOURNAL = {Comm. Contemp. Math.},
FJOURNAL = {Communication in Contemporary Mathematics},
volume = {19},
number = {02},
pages = {1650038, 20 pp},
year = {2017},
doi = {10.1142/S0219199716500383},
}

@article {Ponce2009,
    AUTHOR = {Nahas, J. and Ponce, G.},
     TITLE = {On the persistent properties of solutions to semi-linear
              {S}chr\"{o}dinger equation},
   JOURNAL = {Comm. Partial Differential Equations},
  FJOURNAL = {Communications in Partial Differential Equations},
    VOLUME = {34},
      YEAR = {2009},
    NUMBER = {10-12},
     PAGES = {1208--1227},
      ISSN = {0360-5302},
   MRCLASS = {35Q55 (35B30)},
  MRNUMBER = {2581970},
MRREVIEWER = {Michael J. Goldberg},
       DOI = {10.1080/03605300903129044},
       URL = {https://doi.org/10.1080/03605300903129044},
}

@article {Pausader2007,
    AUTHOR = {Pausader, Benoit},
     TITLE = {Global well-posedness for energy critical fourth-order {S}chr\"{o}dinger equations in the radial case},
   JOURNAL = {Dyn. Partial Differ. Equ.},
  FJOURNAL = {Dynamics of Partial Differential Equations},
    VOLUME = {4},
      YEAR = {2007},
    NUMBER = {3},
     PAGES = {197--225},
      ISSN = {1548-159X},
   MRCLASS = {35Q55 (35B25 35B30)},
  MRNUMBER = {2353631},
MRREVIEWER = {Alessio Pomponio},
       DOI = {10.4310/DPDE.2007.v4.n3.a1},
       URL = {https://doi.org/10.4310/DPDE.2007.v4.n3.a1},
}

@article {Pausader2009_1,
    AUTHOR = {Pausader, Beno\^{\i}t},
     TITLE = {The focusing energy-critical fourth-order {S}chr\"{o}dinger equation with radial data},
   JOURNAL = {Discrete Contin. Dyn. Syst.},
  FJOURNAL = {Discrete and Continuous Dynamical Systems. Series A},
    VOLUME = {24},
      YEAR = {2009},
    NUMBER = {4},
     PAGES = {1275--1292},
      ISSN = {1078-0947},
   MRCLASS = {35Q55 (35B25 35P25)},
  MRNUMBER = {2505703},
MRREVIEWER = {Alessio Pomponio},
       DOI = {10.3934/dcds.2009.24.1275},
       URL = {https://doi.org/10.3934/dcds.2009.24.1275},
}

@article {Pausader2009_2,
    AUTHOR = {Pausader, Benoit},
     TITLE = {The cubic fourth-order {S}chr\"{o}dinger equation},
   JOURNAL = {J. Funct. Anal.},
  FJOURNAL = {Journal of Functional Analysis},
    VOLUME = {256},
      YEAR = {2009},
    NUMBER = {8},
     PAGES = {2473--2517},
      ISSN = {0022-1236},
   MRCLASS = {35Q55 (35B30 35G25 35P25 78A60)},
  MRNUMBER = {2502523},
MRREVIEWER = {Jan L. Cie\'{s}li\'{n}ski},
       DOI = {10.1016/j.jfa.2008.11.009},
       URL = {https://doi.org/10.1016/j.jfa.2008.11.009},
}

@article {Dinh2018,
    AUTHOR = {Dinh, Van Duong},
     TITLE = {On well-posedness, regularity and ill-posedness for the nonlinear fourth-order {S}chr\"{o}dinger equation},
   JOURNAL = {Bull. Belg. Math. Soc. Simon Stevin},
  FJOURNAL = {Bulletin of the Belgian Mathematical Society. Simon Stevin},
    VOLUME = {25},
      YEAR = {2018},
    NUMBER = {3},
     PAGES = {415--437},
      ISSN = {1370-1444},
   MRCLASS = {35Q55 (35B30 35G20 35G25)},
  MRNUMBER = {3852677},
       DOI = {10.36045/bbms/1536631236},
       URL = {https://doi.org/10.36045/bbms/1536631236},
}

@article {Stein1961,
    AUTHOR = {Stein, E. M.},
     TITLE = {The characterization of functions arising as potentials},
   JOURNAL = {Bull. Amer. Math. Soc.},
  FJOURNAL = {Bulletin of the American Mathematical Society},
    VOLUME = {67},
      YEAR = {1961},
     PAGES = {102--104},
      ISSN = {0002-9904},
   MRCLASS = {31.22},
  MRNUMBER = {123729},
MRREVIEWER = {I. I. Hirschman, Jr.},
       DOI = {10.1090/S0002-9904-1961-10517-X},
       URL = {https://doi.org/10.1090/S0002-9904-1961-10517-X},
}

@article{Alex,
    AUTHOR = {Oscar Riano and Alex D. Rodriguez and Svetlana Roudenko},
     title = {Nonlinear {S}chr\"{o}dinger equation with combined nonlinearities},
   JOURNAL = {Nonlinearity},
VOLUME = {39}, 
YEAR = {2026}, 
PAGES = {055014},
DOI = {doi:10.1088/1361-6544/ae64a1},
}

@article{SRR, 
    AUTHOR = {Sanchez, Beckett and Ria\~no, Oscar and Roudenko, Svetlana},
     TITLE = {On the well-posedness of the periodic fractional
              {S}chr\"odinger equation},
   JOURNAL = {Partial Differ. Equ. Appl.},
  FJOURNAL = {Partial Differential Equations and Applications},
    VOLUME = {6},
      YEAR = {2025},
    NUMBER = {3},
     PAGES = {Paper No. 23, 28},
      ISSN = {2662-2963,2662-2971},
   MRCLASS = {35Q55 (35A01 35B10 35Q40 35R11)},
  MRNUMBER = {4899400},
MRREVIEWER = {Luen-Chau\ Li},
       DOI = {10.1007/s42985-025-00327-0},
       URL = {https://doi.org/10.1007/s42985-025-00327-0},
}

@article{KS1991,
title = {Influence of high-order dispersion on self-focusing. II. Numerical investigation},
journal = {Physics Letters A},
volume = {160},
number = {6},
pages = {538-540},
year = {1991},
issn = {0375-9601},
doi = {https://doi.org/10.1016/0375-9601(91)91064-K},
url = {https://www.sciencedirect.com/science/article/pii/037596019191064K},
author = {V.I. Karpman and A.G. Shagalov}
}

@article{Tam2020,
  title = {Generalized dispersion {K}err solitons},
  author = {Tam, Kevin K. K. and Alexander, Tristram J. and Blanco-Redondo, Andrea and de Sterke, C. Martijn},
  journal = {Phys. Rev. A},
  volume = {101},
  issue = {4},
  pages = {043822},
  numpages = {11},
  year = {2020},
  month = {Apr},
  publisher = {American Physical Society},
  doi = {10.1103/PhysRevA.101.043822},
  url = {https://link.aps.org/doi/10.1103/PhysRevA.101.043822}
}

@article{Tam2019,
author = {Kevin K. K. Tam and Tristram J. Alexander and Andrea Blanco-Redondo and C. Martijn de Sterke},
journal = {Opt. Lett.},
number = {13},
pages = {3306--3309},
publisher = {Optica Publishing Group},
title = {Stationary and dynamical properties of pure-quartic solitons},
volume = {44},
month = {Jul},
year = {2019},
url = {https://opg.optica.org/ol/abstract.cfm?URI=ol-44-13-3306},
doi = {10.1364/OL.44.003306}
}

@article{Nature,
    AUTHOR = {Andrea Blanco-Redondo and C. Martijn de Sterke and J. E. Sipe and Thomas F. Krauss and Benjamin J. Eggleton and Chad Husko},
     TITLE = {Pure-quartic solitons},
   JOURNAL = {Nature Communications},
    VOLUME = {7},
      YEAR = {2016},
    NUMBER = {10427},
     PAGES = {1--8},
       DOI = {10.1038/ncomms10427}
}

@book {Tao2006,
    AUTHOR = {Tao, Terence},
     TITLE = {Nonlinear dispersive equations},
    SERIES = {CBMS Regional Conference Series in Mathematics},
    VOLUME = {106},
 PUBLISHER = {AMS, Providence, RI},
      YEAR = {2006},
     PAGES = {xvi+373},
      ISBN = {0-8218-4143-2},
   MRCLASS = {35Q53 (35B35 35P25 35Q55 37K10)},
  MRNUMBER = {2233925},
MRREVIEWER = {Sebastian\ Herr},
       DOI = {10.1090/cbms/106},
       URL = {https://doi.org/10.1090/cbms/106},
}

@article {GrafakosOh2014,
    AUTHOR = {Grafakos, Loukas and Oh, Seungly},
     TITLE = {The {K}ato-{P}once inequality},
   JOURNAL = {Comm. Partial Differential Equations},
  FJOURNAL = {Communications in Partial Differential Equations},
    VOLUME = {39},
      YEAR = {2014},
    NUMBER = {6},
     PAGES = {1128--1157},
      ISSN = {0360-5302,1532-4133},
   MRCLASS = {47G40 (31C15 42B20 46E35)},
  MRNUMBER = {3200091},
MRREVIEWER = {Steven\ George\ Krantz},
       DOI = {10.1080/03605302.2013.822885},
       URL = {https://doi.org/10.1080/03605302.2013.822885},
}

\end{document}